\documentclass{iopjournal}

\usepackage[utf8]{inputenc}
\usepackage[T1]{fontenc}
\usepackage[english]{babel}
\usepackage{amsmath, amssymb, amsthm, mathtools}
\usepackage{geometry}
\usepackage{algorithm}
\usepackage{algpseudocode}
\usepackage{xcolor}
\usepackage{hyperref}
\usepackage{graphicx}
\usepackage{subcaption}
\usepackage{booktabs}
\usepackage{url}
\usepackage{tikz}
\usetikzlibrary{spy}

\newcommand{\spysize}{2.5cm}
\newcommand{\spyimg}[5]{%
  \begin{tikzpicture}[
    spy using outlines={rectangle, magnification=2.5, size=\spysize, connect spies,
      every spy on node/.append style={red, very thick},
      every spy in node/.append style={red, very thick}}
  ]
    \node[anchor=south west, inner sep=0] (img)
      {\includegraphics[width=\linewidth]{#1}};
    \begin{scope}[x={(img.south east)},y={(img.north west)}]
      \spy on (#2,#3) in node at (#4,#5);
    \end{scope}
  \end{tikzpicture}%
}

\newtheorem{theorem}{Theorem}
\newtheorem{lemma}{Lemma}
\newtheorem{proposition}{Proposition}
\newtheorem{remark}{Remark}
\newtheorem{definition}{Definition}
\newtheorem{assumption}{Assumption}

\newcommand{\R}{\mathbb{R}}
\newcommand{\N}{\mathbb{N}}
\newcommand{\norm}[1]{\left\| #1 \right\|}

\newcommand{\dual}[2]{\langle #1, #2 \rangle}
\newcommand{\inner}[2]{\left \langle #1, #2 \right \rangle}
\newcommand{\prox}{\text{prox}}
\newcommand{\dom}{\text{dom}}
\newcommand{\Dsigma}{D_{\sigma}}
\newcommand{\hsigma}{h_{\sigma}}
\newcommand{\phisigma}{\phi_{\sigma}}
\newcommand{\gsigma}{g_{\sigma}}
\newcommand{\id}{\mathrm{Id}}
\DeclareMathOperator*{\argmin}{arg\,min}
\DeclareMathOperator*{\im}{Im}
\newcommand{\meta}{\frac{1}{2}}
\newcommand{\inv}[1]{\frac{1}{#1}}
\newcommand{\tk}[1]{\tilde{#1}^k}

\newcommand{\tkd}[1]{\tilde{#1}_k}

\begin{document}


\title{PnP-IPA: A Provably Convergent Plug-and-Play Inexact Proximal Algorithm for Nonconvex Imaging Problems}

\author{Cristiano Parenti$^{1,*}$\orcid{0009-0006-2780-553X}, Silvia Bonettini$^1$\orcid{0000-0003-2936-365X} and Marco Prato$^{1}$\orcid{0000-0002-7327-3347}}

\affil{$^1$Dipartimento di Scienze Fisiche, Informatiche e Matematiche, Universit\`a di Modena e Reggio Emilia, Via Campi 213/b, 41125 Modena, Italy}

\affil{$^*$Author to whom any correspondence should be addressed.}

\email{cristiano.parenti@unimore.it}


\begin{abstract}
    Plug-and-Play (PnP) methods have emerged as a highly effective paradigm for solving imaging inverse problems by replacing traditional proximity operators of regularization terms with highly expressive deep denoisers. While empirically successful, establishing rigorous convergence guarantees for PnP algorithms remains a major challenge. Existing provable approaches based on the Gradient-Step (GS) denoiser suffer from theoretical and practical limitations, such as restrictive bounds on the regularization parameter, rigid step--size rules, and the inability to handle nonconvex data-fidelity terms. In this paper, we introduce PnP-IPA (Plug-and-Play Inexact Proximal Algorithm), a novel optimization scheme that overcomes these bottlenecks. We propose a new splitting strategy that evaluates the proximal operator of the scaled implicit regularizer inexactly. To enable adaptive step--size selection without exact objective evaluations, we design a novel surrogate merit function that successfully drives an Armijo-like backtracking line--search. Relying on the Kurdyka--{\L}ojasiewicz property, we establish global convergence to a stationary point of the nonconvex objective without imposing any assumption on the regularization parameter. Extensive numerical experiments on image deblurring under both Gaussian and Cauchy noise demonstrate the practical advantages of PnP-IPA. By effectively lifting previous theoretical constraints, our method allows for optimal parameter tuning, yielding state-of-the-art restoration quality and robust convergence even in nonconvex regimes.
\end{abstract}

\section{Introduction}

The restoration of high-quality images from corrupted or incomplete measurements is a fundamental task in modern computational imaging, encompassing applications from medical reconstruction to microscopic and astronomical imaging \cite{Bertero-etal-2022}. Mathematically, the observation process is typically modeled as a linear or nonlinear forward process subject to noise:
\begin{equation*}
    y = \mathcal{N}(A(x)),
\end{equation*}
where $x \in \R^n$ is the unknown underlying image, $y \in \R^m$ represents the noisy observation as a realization of the noise model, $A : \R^n \to \R^m$ is the forward operator (e.g., a blurring kernel, subsampled Fourier transform, or Radon transform), and $\mathcal{N}$ denotes the noise model. 

Since the operator $A$ is often ill-conditioned or possesses a nontrivial null space, recovering $x$ from $y$ is an ill-posed inverse problem. To overcome this, the classical framework relies on Bayesian inference. By Bayes' theorem, the posterior probability of the image given the observation is:
\begin{equation*}
    P(x|y) = \frac{P(y|x)P(x)}{P(y)} \propto P(y|x)P(x).
\end{equation*}
The most common approach to estimate $x$ is to find the Maximum A Posteriori (MAP) estimator, which maximizes $P(x|y)$, or equivalently, minimizes its negative logarithm. This yields the composite optimization problem:
\begin{equation}\label{prob:initial}
    \min_{x \in \R^n} F(x) \equiv f_{data}(x) + \lambda \mathcal{R}(x),
\end{equation}
where $f_{data}(x) = -\log P(y|x)$ ensures data fidelity (e.g., the least-squares $\frac{1}{2}\norm{A x - y}^2$ under Gaussian noise), and $\mathcal{R}(x) = -\log P(x)$ is a regularization term encoding prior knowledge about the statistical distribution of images we wish to recover. The parameter $\lambda > 0$ controls the trade-off between fidelity to the measurements and the prior. Historically, analytical priors were widely used to address this problem, with milestones such as Tikhonov regularization \cite{Engl-etal-1996} and the Total Variation (TV) model \cite{Rudin-etal-1992}.

When the regularizer $\mathcal{R}$ is convex but nonsmooth (as is the case with $\ell_1$-norm sparsity priors or TV), problem \eqref{prob:initial} cannot be solved via standard gradient descent. Instead, it is traditionally addressed via proximal gradient schemes \cite{Combettes-Pesquet-2011,Combettes-Wajs-2005}, which alternate a forward gradient step on the smooth fidelity term with a backward proximal step on the regularizer. The iterative scheme is given by:
\begin{equation}\label{eq:fb_scheme}
    \begin{aligned}
        &z^k = x^k - \alpha_k \nabla f_{data}(x^k),  \\
        &x^{k+1} = \prox_{\alpha_k \lambda \mathcal{R}}(z^k),
    \end{aligned}
\end{equation}
where $\alpha_k > 0$ is the step size and the proximity operator is defined as 
\begin{equation}\nonumber
    \prox_{\tau \mathcal{R}}(z) = \argmin_{x} \left\{ \mathcal{R}(x) + \frac{1}{2\tau}\norm{x - z}^2 \right\}.
\end{equation}
Notably, evaluating the proximity operator is mathematically equivalent to solving a pure denoising problem on $z^k$ with assumed Gaussian noise level proportional to $\sqrt{\tau}$. 

While analytical priors are mathematically tractable, they struggle to capture the complex, high-dimensional manifolds of natural images. Then, in the last decade, the surge in the availability of large amounts of data, the development of statistical learning theory \cite{Hastie-etal-2009}, and the impressive results achieved by neural network models have also influenced the inverse problems community, giving rise to data-driven approaches \cite{Arridge-etal-2019,Bubba-2025}. Examples of these methods are unrolled optimization algorithms \cite{Bertocchi-etal-2020,Gregor-LeCun-2010} and Plug-and-Play (PnP) \cite{Kamilov-etal-2023} schemes. We focus on the latter, which has become a dominant paradigm in computational imaging due to its modularity and empirical performance. 

{\em Framework of the contribution}. The Plug-and-Play (PnP) framework was introduced as a way to leverage the powerful denoising capabilities of modern Deep Neural Networks (DNNs) within the structure of traditional optimization algorithms. In the seminal paper \cite{Venkatakrishnan-etal-2013} it was observed that proximal step in \eqref{eq:fb_scheme} acts simply as an image denoiser. They proposed to replace the proximity operator with an off-the-shelf, highly expressive pre-trained deep denoiser, which we denote as $\Dsigma$. The PnP forward--backward (FB) algorithm simply becomes
\begin{equation}\label{eq:pnp_fb_scheme}
x^{k+1} =\Dsigma( x^k - \alpha_k \nabla f_{data}(x^k)).
\end{equation}
PnP methods have since demonstrated state-of-the-art empirical performance across numerous imaging domains. However, divorcing the denoising step from a strictly defined optimization objective means that the algorithm is no longer guaranteed to minimize a specific global function $F(x)$, making theoretical convergence analysis extremely challenging.

To bridge the gap between empirical success and theoretical guarantees, a major breakthrough was achieved in establishing the existence of regularizers parameterized by neural networks for which the proximal operator is exactly a denoiser. A remarkable result in this framework was taken in \cite{Hurault-etal-2022b}, where the authors introduced the Gradient-Step (GS) denoiser which takes the form of a gradient descent step:
\begin{equation*}
\Dsigma = \id - \nabla \gsigma, \qquad \gsigma = \meta \norm{x - N_\sigma(x)}^2,
\end{equation*}
being $\id$ the identity operator and $N_\sigma$ a parametrized differentiable neural network. Moreover, in the same paper the following FB scheme applied to the functional $f_{data} + \lambda \gsigma$ was proposed, 
\begin{equation*}
      x^{k+1} = \prox_{\tau f_{data}} \left(\tau \lambda \Dsigma(x^{k}) + (1 - \tau \lambda) x^{k}\right)
        = \prox_{\tau f_{data}} \left(x^{k} - \tau \lambda \nabla g_\sigma(x^{k})\right), 
\end{equation*}
where the proximal step is computed with respect to  the (convex) data-fidelity term.
This scheme turns out to be nonpractical in many imaging inverse problems, since the proximal operator of the data-fidelity term is not computable in closed form. In the follow-up work \cite{Hurault-etal-2022a}, given the constraint that $\nabla \gsigma$ is $L_{\gsigma}$--Lipschitz with $L_{\gsigma} < 1$, the same authors proved that the GS denoiser corresponds to the proximity operator of an explicit prox-regular function $\phisigma$, i.e., $D_\sigma = \prox_{\phi_\sigma}$, with $\prox_{\phi_\sigma}$ single valued. Hence, they consider iteration \eqref{eq:pnp_fb_scheme} with $\alpha_k = \lambda$, showing that it corresponds to a FB step applied to the functional $\lambda f_{data} + \phisigma$ with a fixed step size equal to one. The convergence of this PnP-FB method is then proved for possibly nonconvex $f_{data}$ under a constraint on the regularization parameter $\lambda$, which must satisfy $\lambda L_{f_{data}} <1$, where $L_{f_{data}}$ is the Lipschitz constant of $\nabla f_{data}$. Another step further was taken in \cite{Hurault-etal-2024}, where the authors proved that the regularizer $\phisigma$ can be expressed for some constant $K \in \R$ as
\begin{equation} \label{eq:hurault_phi}
    \phisigma(x) = \hsigma^*(x) - \meta\norm{x}^2 + K, \qquad \forall x \in \im(\Dsigma),
\end{equation}
where $\hsigma$ satisfies $\Dsigma=\nabla \hsigma$ and $\hsigma^*$ denotes its convex conjugate. Formula \eqref{eq:hurault_phi} shows that the function $\phisigma$ is weakly convex,  allowing the authors to perform their convergence analysis under the assumption $\lambda L_{f_{data}} < \frac{L_{\gsigma} + 2}{L_{\gsigma} + 1}$. Moreover, by formulating a relaxed version of the PnP-FB scheme, they proved convergence in the convex setting further lightening the constraints on $\lambda$.

Despite the monumental theoretical step provided by the GS denoiser, the derived PnP algorithms based on the FB splitting in \eqref{eq:pnp_fb_scheme} still suffer from major limitations 
\begin{enumerate}
    \item when the data-fidelity is nonconvex, convergence is only guaranteed with the constraint $\lambda L_{f_{data}} < \frac{L_{\gsigma} + 2}{L_{\gsigma} + 1}$, which is a strong limitation in case of low-noise levels and large Lipschitz constants of the data-fidelity gradient;
    \item the constraints can be relaxed only in the convex setting, which limits the applicability of the method, and even in this case computing the exact bound is challenging;
    \item due to the fact that $\prox_{\tau \phisigma}$ is not computable in closed form for an arbitrary scaling factor $\tau \neq 1$, the step--size of the algorithm is fixed and equal to one, which is not optimal for convergence speed.
\end{enumerate}
As an additional difficulty, the function $\phi_\sigma$ can be evaluated in closed form only on the points of the set $\mathrm{Im}(D_\sigma)$. 

{\em Proposed approach}. In this work, we propose a PnP-FB algorithm that overcomes all the limitations mentioned above while preserving the state-of-the-art convergence properties. 
The key point is a new way to split the objective functional $f_{data}+\lambda\phi_\sigma$, based on the dual formulation \eqref{eq:hurault_phi}.
Our splitting requires the computation of $\prox_{\tau h^*_\sigma}$, which is not available in closed form but can be approximated within a given tolerance by an inner procedure based on the evaluation of the denoiser $D_\sigma$. The inner procedure is tailored such that it outputs a point where the function $h^*_\sigma$ can be exactly evaluated. This allows us to design a line--search loop to adaptively compute the step length parameter.

The overall algorithm therefore takes the form of a line--search based inexact FB method, whose convergence analysis builds upon the abstract convergence scheme developed in \cite{Bonettini-etal-2023}. In particular, we prove that any limit point of the iterate sequence is a stationary point of the function $f_{data}+\lambda\phi_\sigma$ and, adopting the classical Kurdyka--Łojasiewicz (KL) \cite{Attouch-etal-2013} framework, we are able to guarantee also global convergence, without any restriction on the regularization parameter $\lambda$ and without requiring convexity of the data-fidelity term.

{\em Related work}. The pursuit of provably convergent PnP methods has been a primary focus of the computational imaging community. Early attempts focused on strictly constraining the Lipschitz constant of the denoiser, as seen in \cite{Ryu-etal-2019}. A parallel approach emerged through Regularization by Denoising (RED) \cite{Romano-etal-2017}, which defined an explicit objective using the denoiser directly, though later works demonstrated that RED acts as a formal gradient only under highly restrictive conditions \cite{Reehorst-Schniter-2018}. More recently, the framework has been extended to denoising diffusion models \cite{Renaud-etal-2024,Song-etal-2020} and flow-matching denoisers \cite{Martin-etal-2025}.

When proximity operators lack closed-form solutions, exact FB algorithms cannot be applied and the proximity operator must be computed inexactly. The foundational convergence rates of inexact proximal-gradient methods were established in \cite{Schmidt-etal-2011} and later generalized in \cite{Villa-etal-2013}. However, allowing inexactness while simultaneously applying a line--search is a nontrivial challenge because the line--search acceptance criteria generally require exact evaluation of the objective function. Crucial progress was made by Bonettini et al. \cite{Bonettini-etal-2020}, who designed variable-metric and line--search strategies for inexact operators. In the last years, this culminated in a generalized abstract KL convergence frameworks \cite{Bonettini-etal-2023} and in the PHILA algorithm \cite{Bonettini-etal-2024}. In our work, we bridge this optimization theory with the PnP imaging paradigm, tailoring the inexact KL framework to absorb the challenges posed by the neural network-based weakly-convex regularizer.

{\em Outline of the paper}. The paper is organized as follows. In Section 2, we present the general formulation of the inverse problem, highlighting the role and properties of the functions involved. In Section 3, we describe in detail our novel line--search-based Plug-and-Play FB algorithm, with particular emphasis on the inexact computation of the proximal point. The convergence analysis of the proposed method is presented in Section 4. Finally, Section 5 reports numerical experiments on image deblurring problems in the presence of different noise statistics, while our final remarks are offered in Section 6.

\section{Formulation of the problem and novel splitting strategy}

As explained in the introduction, we are interested in solving problem \eqref{prob:initial} with a PnP-FB scheme where the proximal step is replaced by a Gradient-Step denoiser $\Dsigma$. Hence, we minimize the functional
\begin{equation}\label{eq:F_first}
    F(x) = f_{data}(x) + \lambda \phisigma(x),
\end{equation}
where $\phisigma$ is defined as in \eqref{eq:hurault_phi}. In the following, we will always make the following assumption:
\begin{assumption} \label{as:data_fidelity}
The data-fidelity term $f_{data}$ is bounded from below, smooth and has a $L_{f_{data}}$--Lipschitz continuous gradient, i.e., 
\begin{equation*}
\norm{\nabla f_{data}(x) - \nabla f_{data}(y)} \leq L_{f_{data}} \norm{x - y}, \quad \forall x, y \in {\rm{dom}}(f_{data}).
\end{equation*}
\end{assumption}
In the following proposition we recall the properties of $\phisigma$.
\begin{proposition}\cite[Proposition 3.1]{Hurault-etal-2022a}, \cite[Proposition 3]{Hurault-etal-2024}  \label{eq:hurault_prop}
Let $\gsigma : \R^n \to \R$ be a $C^{k+1}$ function with $k \geq 1$ such that $\nabla \gsigma$ is $L_{\gsigma}$--Lipschitz continuous, with $L_{\gsigma} < 1$. Let also $\hsigma : \R^n \to \R$ be a $C^{k+1}$ function defined as $h_\sigma(x) = \frac 1 2 \|x\|^2-g_\sigma(x)$ for all $x \in \mathbb{R}^n$. Setting $\Dsigma := \id - \nabla \gsigma = \nabla \hsigma$, then
\begin{itemize}
\item[(i)] $h_\sigma$ is $(1-L_{g_\sigma})$--strongly convex, i.e.,
\begin{equation*}
h_\sigma(y) \geq h_\sigma(x) + \dual{\nabla h_\sigma(x)}{y-x} + \dfrac{1-L_{g_\sigma}}{2}\|y-x\|^2 \qquad \forall x,y \in {\rm{dom}}(h_\sigma).
\end{equation*}
\item[(ii)] $\Dsigma$ is injective, $\mathrm{Im}(\Dsigma)$ is open and, $\forall x \in \R^n$, there exists $\hat{\phi}_\sigma$ such that $\Dsigma = {\rm prox}_{\hat{\phi}_\sigma}$, with 
\begin{equation*}
    \hat{\phi}_\sigma(x) = \begin{cases}
        g_\sigma(D_\sigma^{-1}(x)) - \frac 1 2 \|D_\sigma^{-1}(x)-x\|^2 &\text{if } x\in\mathrm{Im}(D_\sigma)\\
        +\infty & {\text{otherwise}}
    \end{cases}.
\end{equation*}
    \item[(iii)] There exists a $\frac{L_{\gsigma}}{1 + L_{\gsigma}}$- weakly covex function $\phisigma \colon \R^n \to \R$ such that $\Dsigma = {\rm prox}_{\phisigma}$. Moreover, there exists a constant $K \in \R$ such that for all $x\in \mathrm{Im}(D_\sigma)$,
    \begin{equation} \label{eq:phi_def_full}
         \phisigma(x) = \hsigma^*(x) - \frac{1}{2} \norm{x}^2 + K = \hat{\phi}_\sigma(x) + K,
     \end{equation}
     where $h_\sigma^*$ is the Fenchel conjugate of $h_\sigma$.
\end{itemize}
\end{proposition}
The choice of $\gsigma$ is crucial to ensure the existence of the implicit regularizer $\phisigma$ and to obtain convergence guarantees. In particular, it is required that:
\begin{enumerate}
    \item $\gsigma(x) = \meta \norm{x - N_\sigma(x)}^2$, where $N_\sigma$ is a DRUNet with softplus activation functions;
    \item the network is trained to impose that $\nabla \gsigma$ is $L_{\gsigma}$--Lipschitz continuous, with $L_{\gsigma} < 1$.
\end{enumerate}
We refer to \cite{Hurault-etal-2022a} for more details on the training procedure and the choice of the architecture. For our purposes, we focus on the properties of $\phisigma$ and $\hsigma$ that are relevant for the convergence analysis. In particular, we point out that $\Dsigma$ is a Lipschitz continuous operator with constant $L_{D_{\sigma}} = L_{\gsigma} + 1< 2$.

A natural FB approach to minimize the function $F$ in \eqref{eq:F_first} associated to the splitting $f_{data}, \lambda\phi_\sigma$ requires the evaluation of the proximity operator associated to $\tau\phi_\sigma$, where $\tau$ is the product of the regularization parameter $\lambda$ with the stepsize used in the forward step. However, although we can evaluate $\Dsigma = \prox_{\tau \phisigma}$ for $\tau = 1$, we cannot compute $\prox_{\tau \phisigma}$ in closed form for $\tau \neq 1$, because this would require inverting the denoiser $\Dsigma$, which is intractable for Deep Neural Networks. Hence, we propose a novel splitting strategy that circumvents this issue and computes an approximation of $\prox_{\lambda h^*_\sigma}$ via an inner optimization procedure on the primal function $\hsigma$, whose gradient is the denoiser $\Dsigma$. More precisely, in view of \eqref{eq:phi_def_full}, we express the objective function \eqref{eq:F_first} as  
\begin{equation}\label{eq:F}
    F(x) = f_{data}(x) + \lambda \left( h^*_{\sigma}(x) - \meta\norm{x}^2 \right) \qquad \forall x \in \mathbb{R}^n
\end{equation}
 and we split it as follows:
\begin{align}
    f_0(x) &= f_{data}(x) - \frac{\lambda}{2}\norm{x}^2 \quad &\text{(smooth, possibly nonconvex)} \label{eq:f0} \\
    f_1(x) &= \lambda \hsigma^*(x) \quad &\text{(convex)} \label{eq:f1}
\end{align}
Observe that by Assumption \ref{as:data_fidelity}, $f_0$ is smooth with $L_{f_0}$--Lipschitz continuous gradient, where $L_{f_0} = L_{f_{data}} + \lambda$. Moreover, since $h_\sigma$ is strongly convex, $h^*_\sigma$ is differentiable with Lipschitz continuous gradient. Based on the splitting \eqref{eq:f0}--\eqref{eq:f1}, we propose an inexact FB method to minimize the function $F$ in \eqref{eq:F}, as described in the following section.

\begin{remark}\label{remark:bijection}
In general, the objective functional \eqref{eq:F} corresponds to \eqref{eq:F_first} only on $\mathrm{Im}(\Dsigma)$, where $\phisigma$ admits the explicit expression \eqref{eq:phi_def_full}.
 However, under the assumptions of Proposition \ref{eq:hurault_prop}, $\Dsigma$ is a bijection of $\R^n$. Indeed, injectivity follows from the strong convexity of $\hsigma$, which implies the strong monotonicity of $\nabla\hsigma = \Dsigma$. As for surjectivity, fix $x \in \R^n$ and consider $p \mapsto \inner{p}{x} - \hsigma(p)$. By the $(1-L_{\gsigma})$--strong convexity of $\hsigma$, this function is strongly concave, hence coercive towards $-\infty$, and its maximum is attained at a unique $p^* \in \R^n$; the first-order optimality condition reads $x = \nabla\hsigma(p^*) = \Dsigma(p^*)$. Hence $\mathrm{Im}(\Dsigma) = \R^n$ and the two functionals \eqref{eq:F} and \eqref{eq:F_first} coincide in the whole space, up to the additive constant $\lambda K$, which affects neither the minimizers nor the stationary points.
\end{remark}

\begin{remark} \label{remark:surrogate}
The arguments in Remark \ref{remark:bijection}, like all the results in the Gradient-Step denoiser literature \cite{Hurault-etal-2022a,Hurault-etal-2022b,Hurault-etal-2024}, rely on the assumption that $\nabla\gsigma$ is globally $L_{\gsigma}$--Lipschitz with $L_{\gsigma} < 1$. In practice, this constraint is only softly enforced at training time, penalizing the spectral norm of $\nabla^2\gsigma$ on noisy natural images \cite{Hurault-etal-2022a}; for the trained network, it can therefore be expected to hold only approximately. In this respect, the minimization of $F$ in \eqref{eq:F}, which is well defined on the whole $\R^n$ regardless of the above assumption ($\hsigma^*$ being finite everywhere whenever $\hsigma$ is strongly convex), should be regarded as the optimization problem addressed by the proposed algorithm.
\end{remark}

\section{Algorithm overview and inexact proximal computation}

To minimize the objective function $F(x)$ defined in our splitting strategy, we propose an iterative scheme that alternates between gradient steps, inexact proximal updates, and an adaptive line--search. We sketch the algorithmic framework:
\begin{align}
    &z^k = x^k - \alpha_k \nabla f_0(x^k), \label{eq:step_1} \\
    &\tilde{p}^k \approx \prox_{\hsigma/(\alpha_k \lambda )} \left (\frac{z^k}{\alpha_k \lambda} \right ), \label{eq:step_2} \\
    &\tk{y} = \Dsigma(\tilde{p}^k), \label{eq:step_3} \\
    &x^{k+1} = x^k + \eta_k (\tk{y} - x^k), \quad \eta_k \in [0, 1]. \label{eq:step_4}
\end{align}
Before diving into the technical details, we briefly describe the conceptual steps performed at each iteration:
\begin{enumerate}
    \item \textbf{Forward Step \eqref{eq:step_1}.} We perform a gradient step on the smooth part $f_0$ to obtain an intermediate point $z^k$. Our analysis is almost completely independent of the step--size sequence $\{ \alpha_k \}_{k \in \N}$. For theoretical reasons, we must only impose 
\begin{equation}\label{alphalim}
    0 < \alpha_{\text{min}} \leq \alpha_k \leq \alpha_{\text{max}}, \quad  \forall k \in \N,
\end{equation}
for some positive constants $\alpha_{\text{min}}, \alpha_{\text{max}}$. In practice, $\alpha_k$ can be tuned to improve convergence speed.
    \item \textbf{Inexact Backward Step \eqref{eq:step_2}--\eqref{eq:step_3}.} We compute an inexact proximal update $\tilde{p}^k$ for the convex part $f_1$ by shifting the computation to the primal function $\hsigma$ via the Moreau decomposition. This step is the core focus of Sections \ref{sec:3.1}--\ref{sec:3.2}.
    \item \textbf{Line--search Step \eqref{eq:step_4}.} We verify if the step from $x^k$ to $\tk{y}$ provides a sufficient decrease of the objective. If not, we backtrack the step--size and repeat. Since we cannot evaluate $f_1$ exactly, we introduce a novel surrogate merit function. This procedure is detailed in Section \ref{sec:3.3}. 
\end{enumerate}

\subsection{Proximal calculation via inner minimization}\label{sec:3.1}

As already mentioned, we employ the Moreau decomposition \cite[Theorem 14.3, (iii)]{Bauschke-Combettes-2011} to shift the computation from the conjugate $\hsigma^*$ to the primal $\hsigma$:
\begin{equation}\label{eq:hatyk}
    \hat{y}^k := \prox_{\alpha_k \lambda \hsigma^*}(z^k) = z^k - \alpha_k \lambda \prox_{\hsigma/(\alpha_k \lambda )} \left (\frac{z^k}{\alpha_k \lambda} \right ) = z^k - \alpha_k \lambda \hat{p}^k,
\end{equation}
where
\begin{equation} \label{eq:inner_problem}
\hat{p}^k := \prox_{\hsigma/(\alpha_k\lambda)}\left (\frac{z^k}{\alpha_k \lambda} \right ) = \argmin_{p \in \mathbb{R}^n} \underbrace{\frac{1}{2}\norm{ p - \frac{z^k}{\alpha_k \lambda}}^2 + \frac{1}{\alpha_k\lambda} \hsigma(p)}_{:=G_k(p)}.
\end{equation}
Since $\hsigma$ is $(1 - L_{\gsigma})$--strongly convex, the function $G_k$ in \eqref{eq:inner_problem} is $\mu_{G_k}$--strongly convex with
\begin{equation}\label{eq:muG}
\mu_{G_k} = 1 + \frac{1-L_{g_{\sigma}}}{\alpha_k\lambda}\geq 1.
\end{equation}
Moreover, the relation $\Dsigma = \nabla \hsigma$ leads to an explicit formula for its gradient:
\begin{equation}\label{eq:gradG}
    \nabla G_k(p) =  p - \frac{z^k}{\alpha_k \lambda} + \frac{1}{\alpha_k \lambda} \Dsigma(p) \qquad \forall p \in \mathbb{R}^n.
\end{equation}
Hence, an approximate solution $\tilde{p}^k$ of \eqref{eq:inner_problem} can be computed e.g. by performing a gradient descent on $G_k$
arrested when relation
\begin{equation}\label{eq:stop_crit_G}
\| \nabla G_k(\tilde{p}^k) \| \leq \varepsilon_k
\end{equation}
is satisfied, where $\varepsilon_k$ is a suitable positive tolerance.

As concerns the computation of an approximate solution $\tilde{y}^k$ of \eqref{eq:hatyk}, we remark that the straightforward approach of applying the explicit formula $\tilde{y}^k = z^k - \alpha_k \lambda \tilde{p}^k$ in \eqref{eq:hatyk} is not practical, since $\tilde{y}^k$ would not be assured to belong to $\mathrm{Im}(\Dsigma)$; therefore, we would be missing an explicit way to compute $f_1(\tilde{y}^k)$, which, as we will see later, is needed to define a sufficient decrease condition. To overcome this issue, first we observe that from equations \eqref{eq:hatyk}--\eqref{eq:gradG} and since $\nabla G_k(\hat{p}^k) = 0$ one obtains that
\begin{equation}
    \nabla \hsigma (\hat{p}^k) = \hat{y}^k.
\end{equation}
Therefore, the approximation $\tilde{y}^k$ can be computed as 
\begin{equation} \label{eq:ytildeDenoised}
    \tilde{y}^k = \nabla \hsigma (\tilde{p}^k) = \Dsigma(\tilde{p}^k).
\end{equation}
Using this new formulation, we show that we can control the distance between the exact prox and its approximation:
\begin{equation}\nonumber
    \norm{\tilde{y}^k - \hat{y}^k} = \norm{\Dsigma ( \tilde{p}^k ) - \Dsigma ( \hat{p}^k )} \leq L_{D_{\sigma}} \norm{ \tilde{p}^k - \hat{p}^k } \leq \frac{L_{D_{\sigma}}}{\mu_{G_k}} \norm{\nabla G_k(\tilde{p}^k)} \leq \frac{L_{D_{\sigma}}}{\mu_{G_k}}\varepsilon_k,
\end{equation}
where the second inequality follows from the strong convexity of $G_k$ and the last one from the stopping criterion \eqref{eq:stop_crit_G}. Using the fact that $\mu_{G_k} > 1$ and $L_{D_\sigma} < 2$, this leads to
\begin{equation}\label{eq:distyk_ine}
    \norm{\tilde{y}^k - \hat{y}^k} \leq 2\varepsilon_k.
\end{equation}
In addition, if \eqref{eq:ytildeDenoised} holds we are able to explicitly compute $f_1(\tilde y^k)$ and $h_\sigma(\tilde p^k)$.
Indeed, from \eqref{eq:ytildeDenoised} and \eqref{eq:phi_def_full}, we derive the following explicit formula
\begin{align*}
    f_1(\tilde{y}^k) = \lambda \hsigma^*(\tilde{y}^k) &= \lambda\Big( g_\sigma(D_\sigma^{-1}(\tilde y^k))-\frac 1 2 \|D_\sigma^{-1}(\tilde y^k) -\tilde y^k\|^2 +\frac1 2 \|\tilde y^k\|^2 \Big)\\
    &= \lambda \Big( g_\sigma(\tilde p^k) +\dual{\tilde{y}^k}{\tilde{p}^k} - \frac 1 2 \|\tilde p^k\|^2 \Big).
\end{align*}
Moreover, by the Fenchel--Young inequality and the relation \eqref{eq:ytildeDenoised} we obtain \cite[Proposition 16.9]{Bauschke-Combettes-2011}
\begin{equation}\nonumber
  \hsigma^*(\tilde{y}^k) =   \dual{\tilde{y}^k}{\tilde{p}^k} - \hsigma  (\tilde{p}^k)
\end{equation}
which, combined with the previous relation, yields $h_\sigma (\tilde p^k)= \frac 1 2 \|\tilde p^k\|^2 - g_\sigma(\tilde p^k)$.


\subsection{Properties of the inexact proximal gradient point}\label{sec:3.2}

In this section, we state some crucial properties of the point $\tilde y^k$ obtained as described in the previous section. To this end, we follow the notation in \cite{Bonettini-etal-2024} and observe that $\hat{y}^k = \prox_{\alpha_k f_1}(z^k)$ is the unique point such that 
\begin{equation}\label{eq:Hkfirstdef}
\hat{y}^k = \argmin_{y \in \R^n} H_k(y), \qquad H_k(y) = \dual{\nabla f_0(x^k)}{y - x^k} + \frac{1}{2 \alpha_k } \norm{y - x^k}^2 + f_1(y) - f_1(x^k).
\end{equation}
Notice that $H_k(x^k) = 0$, therefore $
H_k(\hat y^k) \leq 0$, and $H_k(\hat y^k) < 0$ whenever $x^k\neq \hat y^k$. On the other hand, $x^k=\hat y^k$ implies that $x^k$ is a stationary point for the function $F$, i.e., $\nabla F(x^k) = 0$. In general, when $f_1$ is convex, $H_k$ is $\theta$--strongly convex, with $\theta =  1/\alpha_{\max}$. Moreover, since in our case $\hsigma$ is $(1 - L_{\gsigma})$--strongly convex, the gradient of $f_1=\lambda \hsigma^*$ is Lipschitz continuous with constant $\frac{\lambda}{1 - L_{\gsigma}}$. Then, the gradient of $H_k$ is $\mu$--Lipschitz continuous with $\mu \geq 1/{\alpha_{\min}} + \lambda/(1-L_{g_\sigma
})$.  

The function $H_k$ plays a crucial role in our algorithm since if $H_k(\tilde y ^k) < 0$, the vector $d^k = \tilde y^k - x^k$ computed in \eqref{eq:step_3} is a descent direction for the objective function $F$ \cite[Proposition 2.2]{Bonettini-etal-2016}. However, the negative sign of $H_k(\tilde y^k)$ alone does not guarantee suitable properties of the descent direction $d^k$, needed to perform the convergence analysis of iteration \eqref{eq:step_1}--\eqref{eq:step_4}. Such properties are actually obtained if the value $H_k(\tilde y^k)$  is sufficiently close to the optimal one, $H_k(\hat y^k)$. For this reason, it is important to estimate the distance from the optimal value. In our settings, 
if $\tilde y^k$ is the point defined in \eqref{eq:ytildeDenoised}, with $\tilde p^k$ satisfying \eqref{eq:gradG}, the $\mu$--smoothness of $H_k$ and the Descent Lemma imply
\begin{equation} \label{eq:ineqH}
    H_k (\tilde{y}^k) - H_k (\hat{y}^k) \leq \frac{\mu}{2} \norm{\tilde{y}^k - \hat{y}^k}^2 \leq 2{\mu}{\varepsilon}_k^2.
\end{equation} 
Hence, up to some constants, we can control the distance $H_k (\tilde{y}^k) - H_k (\hat{y}^k)$ by means of the tolerance parameter $\varepsilon_k$, without knowing the optimal value $H_k (\hat{y}^k)$. Before giving more details on how the tolerance parameter $\varepsilon_k$ is chosen, we explain our line--search approach, because the two issues are closely related.



\subsection{Line--search procedure}\label{sec:3.3}

To ensure convergence in the nonconvex setting, we employ an Armijo-like line search. Considering the descent direction $d^k = \tilde{y}^k - x^k$, the new point is set as 
\begin{equation}
    x^{k+1} = x^k + \eta_k d^k, \quad \eta_k \in (0,1],
\end{equation}
where the line--search parameter $\eta_k$ must satisfy a sufficient decrease condition. A very common example of such conditions is based on the objective function values and writes as 
\begin{equation} \label{eq:armijo0}
    F (x^{k}+\eta_k d^k) \leq F (x^k) + \omega \eta_k \Delta_k,
\end{equation}
where $\Delta_k$ is a negative quantity expressing the expected decrease and $\omega\in(0,1)$ is a fixed parameter \cite{Bonettini-etal-2024}. The value of $\eta_k$ is typically found with a backtracking procedure, i.e., by defining $\eta_k = \delta^{m_k}$, with $\delta \in (0,1)$ and $m_k$ equal to the smallest non--negative integer for which \eqref{eq:armijo0} is satisfied. Moreover, borrowing the ideas in \cite{Bonettini-etal-2016,Bonettini-etal-2020,Bonettini-etal-2021} the sufficient decrease quantity is usually set as $\Delta_k = H_k(\tilde y^k)$. However, this approach cannot be directly adopted in our settings, due to the lack of an explicit way to compute $f_1(x^k +\eta d^k)$ except for the case $\eta = 1$, which is needed to obtain both the quantities $F(x^k+\eta d^k)$ and $H_k(\tilde y^k)$. We overcome this difficulty by introducing an upper estimate for $f_1(x^k)$ in the following way. 
From the convexity of $f_1$, for $\eta \in (0,1]$ it holds that
\begin{equation}\nonumber
    f_1(x^k + \eta d^k) \leq \eta f_1(\tilde{y}^k) + (1 - \eta) f_1(x^k).
\end{equation}
Initializing $U^0 = f_1(x^0)$ for some initial guess $x^0$, at each step we can obtain an upper estimate of $f_1(x^{k+1})$ as follows
\begin{equation}
U^{k+1} := \eta_k f_1(\tilde{y}^k) + (1 - \eta_k) U^k.
\end{equation}
 We next show that the sequence $\{U^k\}_{k\in\mathbb{N}}$ provides an upper estimate of $f_1(x^k)$ for all $k$.
 

\begin{proposition}\label{prop:Uk}
    For all $k \in \N$ it holds that
    \begin{equation}
        f_1(x^k) \leq U^k.
    \end{equation}
\end{proposition}

\begin{proof}
    The proof is obtained by induction. The inequality is satisfied as an equality for $k = 0$. Then, suppose that $f_1(x^k) \leq U^k$. We compute
    \begin{align*}
        U^{k+1} = \eta_k f_1(\tilde{y}^k) + (1 - \eta_k) U^k &\geq \eta_k f_1(\tilde{y}^k) + (1 - \eta_k) f_1(x^k) \\
        &\geq f_1(\eta_k \tk{y} + (1 - \eta_k) x^k) = f_1(x^{k+1}),
    \end{align*}
    which corresponds to the inductive step.
\end{proof}
Once we have defined the sequence $\{ U^k \}_{k \in \N}$, we introduce a new figure of merit as follows
\begin{equation}\label{eq:Phi_def}
\Phi(x^k, U^k) = f_0(x^k) + U^k
\end{equation}
and we define a computable approximate value for $\Delta_k=H_k(\tilde y^k)$ as 
\begin{equation}\label{eq:def_tildeDeltak}
    \tilde{\Delta}_k = \dual{\nabla f_0(x^k)}{\tilde{y}^k - x^k} + \frac{1}{2 \alpha_k } \norm{\tilde{y}^k - x^k}^2 + f_1(\tilde{y}^k) - U^k,
\end{equation} 
which is, in practice, $H_k(\tilde y^k)$ with $f_1(x^k)$ replaced by $U^k$. 
Clearly, the approximation \eqref{eq:def_tildeDeltak} satisfies the inequality
\begin{equation} \label{eq:deltatildeineq}
    \tilde{\Delta}_k \leq H_k(\tilde y^k).
\end{equation}
Based on the previous definitions, we set the backtracking loop to compute the smallest non--negative integer $m_k$ such that
\begin{equation}\label{eq:armijo}
   \min\{\Phi (x^{k}+\delta^{m_k} d^k, \delta^{m_k}f_1(\tilde y^k)+ (1-\delta^{m_k})U^k), F(\tilde y^k) \}\leq \Phi (x^k,U^k) + \omega \delta^{m_k} \tilde \Delta_k.
\end{equation}
The well posedness of the above condition is stated in Lemma \ref{lemma:linesearch} to follow. Its proof relies on a specific choice of the tolerance parameter $\varepsilon_k$ in \eqref{eq:stop_crit_G}, which is detailed in the next section. Finally, we adopt the following criterion to select the actual steplength parameter to be used in the updating rule \eqref{eq:step_4}
\begin{equation}\label{eq:updating_complete}
    \eta_k = 
    \begin{cases}
        \delta^{m_k} & \mbox{ if } \Phi (x^{k}+\delta^{m_k} d^k, \delta^{m_k}f_1(\tilde y^k)+ (1-\delta^{m_k})U^k)< F(\tilde y^k)\\
        1 & \mbox{ otherwise}
    \end{cases}.
\end{equation}
The above choice is motivated by the fact that it implies the following two properties on the point $x^{k+1}$:
\begin{equation}\label{eq:line_search_cons}
    \Phi(x^{k+1},U^{k+1} )\leq \Phi(x^k,U^k) + \omega \delta^{m_k}\tilde\Delta_k\ \  \mbox{ and } \ \ \Phi(x^{k+1},U^{k+1} )\leq F(\tilde y^k),
\end{equation}
which are both needed in the convergence analysis performed in Section \ref{sec:convergence}. 

\subsection{Choice of the tolerance ${\varepsilon}_k$}

In this section, we explain how the tolerance parameter can be chosen in order to: (a) make the vector $d^k$ a descent direction for the merit function $\Phi$; (b) guarantee that the sufficient decrease condition \eqref{eq:armijo} is satisfied for all sufficiently large values of $m_k$. 
Instead of selecting a prefixed sequence $\{{\varepsilon}_k \}_{k \in \N}$, we borrow the ideas from \cite{Bonettini-etal-2016,Bonettini-etal-2024,Villa-etal-2013} and adopt an adaptive strategy that relates $\epsilon_k$ with $\tilde{\Delta}_k$. In fact, we formally set
\begin{equation}\nonumber
    {\varepsilon}_k = \frac{\sqrt{- \tkd{\Delta}}}{2}.
\end{equation}
The above setting must be interpreted in the sense of Algorithm \ref{algo:inner}, where the tolerance value is dynamically adjusted through iterations performed to construct $\tilde y^k$. More precisely, Algorithm \ref{algo:inner} consists of gradient descent steps applied to minimize the function $G_k$ in \eqref{eq:inner_problem}, which generate a sequence of dual vectors $\{p^{k,j}\}_{j\in\N}$ and the corresponding approximate proximal gradient points $\{\tilde y^{k,j} := D_\sigma(p^{k,j})\}_{j\in\N}$.
The inner steps are then stopped at iteration $j$ so that 
\begin{equation*}
    \norm{\nabla G_k({p^{k,j}})}^2 \leq - \frac{\tilde{\Delta}_{k,j}}{4} = - \frac{1}{4} \Big(\dual{\nabla f_0(x^k)}{\tilde{y}^{k,j} - x^k} + \frac{1}{2 \alpha_k } \norm{\tilde{y}^{k,j} - x^k}^2 + \lambda h^*(\tilde{y}^{k,j}) - U^k \Big).
\end{equation*}
To show that the above inequality is well posed, we first notice that $\norm{\nabla G_k(p^{k,j})} \underset{j\to\infty}{\longrightarrow} \norm{\nabla G_k(\hat{p}^k)} = 0$, while, if $x^k$ is not a stationary point so that $H_k(\hat y^k)<0$, by continuity we also have
\begin{equation*}
H_k(\tilde y^{k,j}) \underset{j\to\infty}{\longrightarrow} H_k (\hat{y}^k) < 0.
\end{equation*}
Since $\tilde{\Delta}_{k,j}\leq H_k(\tilde y^{k,j})$, for sufficiently large $j$, we have $\tilde \Delta_{k,j} < 0$. 
Therefore, the inner stopping rule is well defined and Algorithm \ref{algo:inner} outputs $\tilde y^k$ and the corresponding value $\tilde \Delta_k < 0$ such that $\|\nabla G(\tilde p^k)\|^2 \leq -\tilde \Delta_k/4 $. Combining this inequality with \eqref{eq:ineqH}, we obtain
\begin{equation}\label{eq:basicHYtilde}
H_k(\tk{y}) - H_k(\hat{y}^k) \leq - \frac{\mu}{2} \tkd{\Delta}.
\end{equation}
We observe that our arguments still hold if replace the inner gradient step in Algorithm \ref{algo:inner} with any other converging optimization method, e.g. accelerated or inertial gradient methods, since
\eqref{eq:basicHYtilde} depends only on the termination rule.

\subsection{PnP-IPA pseudocode}
We summarize the whole proposed algorithm in the following pseudocode. The main routine is PnP-IPA (which stands for Plug-and-Play Inexact Proximal Algorithm) and the inner routine is inexact\_den\_prox, which computes the inexact proximal update via an inner optimization method on the primal function $\hsigma$.

\begin{algorithm}[H]
\caption{PnP-IPA: Plug-and-Play Inexact Proximal Algorithm}\label{algo:pnpipa}
\begin{algorithmic}[1]
\State \textbf{Input:} $x_{\text{init}}$, $\lambda$, $\delta$, $\omega$, $\{\alpha_k \}_{k\in\mathbb{N}}$
\State Initialization: 
\State \qquad $x^0 = \Dsigma(x_{\text{init}})$
\State \qquad $U^0 = f_1(x^0) = \lambda (\dual{x^0}{x_{\text{init}}} - \hsigma(x_{\text{init}}) )$
\For{$k=0, 1, \dots$}
    \State \textbf{1. Forward Step (Gradient):}
    \State Compute $z^{k} = x^{k} - \alpha_k \nabla f_0(x^{k})$, where $\nabla f_0(x) = \nabla f_{data}(x) - \lambda x$
    \State \textbf{2. Backward Step (Inexact Prox via Inner Loop):}
    \State $\tilde{y}^{k}, f_1(\tilde{y}^{k}), \tilde{\Delta}_k =  \mathrm{inexact\_den\_prox}(z^{k}, \nabla f_0(x^{k}), U^k, x^k, \alpha_k, \lambda)$
    \State \textbf{3. Line search:}
    \State Compute $d^{k} = \tilde{y}^{k} - x^{k}$
    \State Set $\Phi_k = f_0(x^k) + U^k$
    \State \textbf{Find} the largest $\eta \in \{1, \delta, \delta^2, \dots\}$ such that
    \State \qquad $U^{k+1} = \eta f_1(\tilde{y}^k) + (1-\eta)U^k$
    \State \qquad $\min \{ \Phi(x^k + \eta d^k, U^{k+1}), f_0(\tilde{y}^k) + f_1(\tilde{y}^k) \} \le \Phi_k + \omega \eta \tilde{\Delta}_k$
    \State \textbf{if} $f_0(\tilde{y}^k) + f_1(\tilde{y}^k) \leq \Phi(x^k + \eta d^k, U^{k+1})$
    \State \qquad Set $x^{k+1} = \tilde{y}^k$, $ U^{k+1} = f_1(\tilde{y}^k)$
    \State \textbf{else}
    \State \qquad Set $x^{k+1} = x^k + \eta d^k$, $U^{k+1} = U^{k+1}$
\EndFor
\end{algorithmic}
\end{algorithm}

\begin{algorithm}[H]
\caption{inexact\_den\_prox inner routine}\label{algo:inner}
\begin{algorithmic}[1]
\State \textbf{Input:} $z^{k}, \nabla f_0(x^{k}), U^k, x^k, \alpha_k, \lambda$, $\zeta$
\State Set $p^{k,0} = \frac{z^k}{\alpha_k \lambda}$
\For{$j=0, 1, \dots$}
\State Compute $\tilde{y}^{k,j} = \Dsigma (p^{k,j})$, $f_1(\tilde{y}^{k,j}) = \lambda (\dual{p^{k,j}}{\tilde{y}^{k,j}} - \hsigma(p^{k,j}) )$
\State Compute ${\tilde \Delta}_{k,j} = \Big( \dual{\nabla f_0(x^k)}{\tilde{y}^{k,j} - x^k} + \frac{1}{2 \alpha_k } \norm{\tilde{y}^{k,j} - x^k}^2 + f_1(\tilde{y}^{k,j}) - U^k \Big)$
    \If{$\norm{\nabla G_k(p^{k,j})}^2 \leq -{\tilde \Delta}_{k,j} / 4$}
    \State break
    \Else
    \State Compute $p^{k,j+1} = p^{k,j} - \zeta\nabla G_k(p^{k,j}) = p^{k,j} - \zeta \left(p^{k,j} - \frac{1}{\alpha_k \lambda} (z^k -\tilde y^{k,j})\right)$ 
    \EndIf
\EndFor
\State \textbf{return:} $\tilde y^k\leftarrow \tilde{y}^{k,j}, f_1(\tilde y^k)\leftarrow f_1(\tilde{y}^{k,j})$, $\tilde \Delta_k\leftarrow {\tilde \Delta}_{k,j}$
\end{algorithmic}
\end{algorithm}

\section{Convergence analysis}\label{sec:convergence}

In this section, we perform the convergence analysis for Algorithm \ref{algo:pnpipa} by adopting a more general point of view. In particular, we refer to any composite optimization problem of the form 
\begin{equation}\label{eq:problem}
\min_{x\in\R^n} F(x), \ \ \mbox{ with } F(x)=f_0(x) + f_1(x),
\end{equation}
 under the following assumptions on the functions $f_0, f_1$:
\begin{description}
    \item[{[A1]}] $f_1 : \R^n \to {\R}\cup\{+\infty\}$ is a proper, lower semicontinuous, convex function.
    \item[{[A2]}] $f_0 : \R^n \to \R$ is continuously differentiable on an open set $\Omega_0 \supset \overline{\dom(f_1)}$.
    \item[{[A3]}] $f_0$ has $L_{f_0}$--Lipschitz continuous gradient on $\dom(f_1)$.
    \item[{[A4]}] $F$ is bounded from below.
\end{description}
Clearly, the specific settings in \eqref{eq:f0}--\eqref{eq:f1} satisfy [A1]--[A4]. For nonconvex nonsmooth optimization problems, the stationary points are characterized as follows.
\begin{definition}\label{def:optcond}
    We say that a point $x\in\R^n$ is stationary for $F:\R^n\rightarrow \R\cup\{+\infty\}$ if $x\in\dom(F)$ and $0\in\partial F(x)$, where $\partial F(x)$ is the (limiting) subdifferential of $F$ at $x$ in \cite[Definition 8.3]{Rockafellar-Wets-1998}.
\end{definition}

Any (local) minimum point of $F$ is a stationary point, while the converse, in general, is not true, except the case where $f_1$ is also convex.
The subdifferential of the objective function of problem \eqref{eq:problem} under assumption [A2] can be expressed as stated in the following Lemma.  
\begin{lemma}\label{lem:subcalculus_basic} \cite[Exercise 8.8c]{Rockafellar-Wets-1998}
If $F=f_0+f_1$ with $f_1$ finite at $x$ and $f_0$ continuously differentiable on a neighbourhood of $x$, then
        $$
        \partial F(x) = \nabla f_0(x)+\partial f_1(x).
        $$
\end{lemma}
In particular, for convex functions, the definition of subdifferential reduces as follows.
\begin{definition}\label{def:subdiff} 
    Let $f_1:\R^n\rightarrow \R\cup\{+\infty\}$ be a proper convex function. Given $x\in\R^n$, the subdifferential of $f_1$ at $x$ is the set
    \begin{equation*}
        \partial f_1(x) = \{v\in\R^n: \ f_1(y)\geq f_1(x)+\langle v, y-x\rangle, \ \forall \ y\in\R^n\}.
    \end{equation*}
\end{definition}

In the sequel, instead of focusing on Algorithm \ref{algo:pnpipa}, we will analyze the convergence from an abstract point of view, i.e., through the properties of the sequences of the iterates. Therefore, we consider any sequences $\{x^k\}_{k\in\N}$, $\{\tilde y^k\}_{k\in\N}$, $\{\hat y^k\}_{k\in\N}$, $\{U_k\}_{k\in\N}$ where $x^{k+1} = x^k+\eta_k(\tilde y^k-x^k)$, with $\eta_k\in (0,1]$ and the following holds:
\begin{description}
    \item[{[S1]}] $U^0=f_1(x^0)$ and $U^{k+1} = \eta_kf_1(\tilde y^k) + (1-\eta_k) U^k$ for $k\geq 0$;
    \item[{[S2]}] $\hat y^k = \prox_{\alpha_kf_1}(x^k-\alpha_k\nabla f_0(x^k))$ and $H_k(\tilde y^k) - H_k(\hat y^k) \leq -\frac \mu 2 \tilde\Delta_k$ for some $\mu > 0$, with $\tilde \Delta_k$ defined as in \eqref{eq:def_tildeDeltak} and $\alpha_k\in [\alpha_{\min},\alpha_{\max}]$; 
    \item[{[S3]}]$\eta_k$ is defined as in \eqref{eq:armijo}--\eqref{eq:updating_complete}, with $\Phi$ as in \eqref{eq:Phi_def}.
\end{description}
The sequences generated by Algorithm \ref{algo:pnpipa} clearly satisfy [S1]--[S3]. 
\subsection{Preliminary results and well posedness of the line--search}
Now we start our analysis by three preliminary lemmata, whose proofs can be found in Appendix \ref{sec:proofs}.

\begin{lemma}[Squared norm estimates] \label{lem:squaredNorms}
    Assume that $f_0,f_1$ satisfy [A1]--[A2] and that [S1]--[S2] hold for the sequences $\{x^{k}\}_{k \in \N}$, $\{\tilde{y}^{k}\}_{k \in \N}$, $\{\hat{y}^{k}\}_{k \in \N}$, $\{U^k\}_{k\in\N}$. Then, the following inequalities hold true:
    \begin{eqnarray} 
        \norm{\tilde{y}^k - \hat{y}^k}^2 &\leq& -\frac\mu\theta \tkd{\Delta} \label{eq:1epsilon}\\
        \norm{\hat{y}^k - x^k}^2 &\leq& - \frac{2}{\theta} \left ( 1 + \frac{\mu}{2} \right ) \tkd{\Delta}\label{eq:2epsilon}\\
        \norm{\tilde{y}^k - x^k}^2 &\leq&  - \frac 4\theta (\mu+1) \tkd{\Delta} \label{eq:3epsilon}
    \end{eqnarray}
where $\mu,\theta$ are positive parameters. 
\end{lemma}

\begin{lemma} \label{lem: inequalities}
Suppose Assumptions [A1]--[A3] hold true and that [S1]--[S2] hold for the sequences $\{x^{k}\}_{k \in \N}$, $\{\tilde{y}^{k}\}_{k \in \N}$, $\{\hat{y}^{k}\}_{k \in \N}$, $\{U^k\}_{k\in\N}$. Then, there exists $c,d>0$ depending only on $L_{f_0}, \theta, \mu, \alpha_{\min}$, such that
\begin{align}
    F(\hat{y}^k) &\geq F(\tilde{y}^k) + c \tkd{\Delta}, \label{eq:firstLem2} \\
    F(\hat{y}^k) &\leq F(x^k) - d \tkd{\Delta}. \label{eq:secondLem2}
\end{align}
\end{lemma}

\begin{lemma} \label{lem:subgrad_bound}
 Assume that $f_0,f_1$ satisfy [A1]--[A3] and that [S1]--[S2] hold for the sequences $\{x^{k}\}_{k \in \N}$, $\{\tilde{y}^{k}\}_{k \in \N}$, $\{\hat{y}^{k}\}_{k \in \N}$, $\{U^k\}_{k\in\N}$. Then, the vectors $\hat{w}^{k} =  -\nabla f_0(x^k) -\frac{1}{\alpha_k}(\hat y^k-x^k)$, $ \hat v^k = \hat w^k + \nabla f_0(\hat y^k)$ satisfy $\hat w^k \in  \partial f_1(\hat y^k) $ and 
\begin{equation} \label{eq:subgrad_ineq}
\hat{v}^{k} \in \partial F(\hat{y}^{k})\ \ \mbox{ and }\ \ \norm{\hat{v}^{k}} \leq q \sqrt{- \tkd{\Delta}},
\end{equation}
where the constant $q$ depends only on $L_{f_0}, \theta, \mu, \alpha_{\min}$.
\end{lemma}

Next, we show that the line search procedure is well-defined and that the parameter $\eta_k$ computed at each step is bounded from below by a positive constant depending only on $L_{f_0}, \theta, \mu$.

\begin{lemma}[Finite termination of the line search] \label{lemma:linesearch}
Assume that $f_0,f_1$ satisfy [A1]--[A3] and  that [S1]--[S3] hold for the sequences $\{x^{k}\}_{k \in \N}$, $\{\tilde{y}^{k}\}_{k \in \N}$, $\{\hat{y}^{k}\}_{k \in \N}$, $\{U^k\}_{k\in\N}$.
Then
\begin{enumerate}
    \item The inequality \eqref{eq:armijo} is well-posed.
    \item There exists $\eta_{\min} > 0$ depending only on $L_{f_0}, \theta, \mu$, such that the parameter $\eta_k$ computed with \eqref{eq:armijo}--\eqref{eq:updating_complete} satisfies
    \begin{equation} \label{eq:eta_min}
    \eta_k \geq \eta_{\min}, \quad \forall k \geq 0.
    \end{equation}
\end{enumerate}
\end{lemma}

\begin{proof}
Let $\eta \in (0, 1]$ and $d^k = \tilde y^k - x^k$. Then
\begin{align*}
f_0(x^{k} + \eta d^k)  & + (1 - \eta) U^k + \eta f_1(\tilde{y}^{k})
\leq f_0(x^{k}) + \eta \inner{\nabla f_0(x^{k})}{d^{k}} + L_{f_0} \frac{\eta^2}{2} \norm{d^{k}}^2 + (1 - \eta) U^k + \eta f_1(\tilde{y}^{k}) \\
&= f_0(x^{k}) + U^k + \eta \left( f_1(\tilde{y}^{k}) - U^k + \inv{\alpha_k} \inner{\alpha_k \nabla f_0(x^{k})}{d^{k}} \right) + L_{f_0} \frac{\eta^2}{2} \norm{d^{k}}^2 \\
&= f_0(x^{k}) + U^k + \eta \Big(  \tilde{\Delta}_k - \inv{2 \alpha_k} \norm{d^k}^2 \Big) + L_{f_0} \frac{\eta^2}{2} \norm{d^{k}}^2 \\
&\leq f_0(x^{k}) + U^k + \eta  \tilde{\Delta}_k + L_{f_0} \frac{\eta^2}{2} \norm{\tk{y} - x^k}^2 \\
&\leq f_0(x^{k}) + U^k + \eta \tilde{\Delta}_k  - \eta^2 L_{f_0} \left( \frac 2 \theta (\mu+1) \right) \tkd{\Delta},
\end{align*}
where the first inequality follows from the Descent Lemma and the third one from \eqref{eq:3epsilon}. Setting $\kappa = L_{f_0} \left( \frac 2 \theta (\mu+1) \right)$, the last inequality writes as
\begin{equation*}
\Phi(x^k+\eta d^k, (1 - \eta) U^k + \eta f_1(\tilde{y}^{k}))\leq \Phi(x^{k}, U^k) + \eta (1 - \kappa \eta) \tkd{\Delta}.
\end{equation*}
Since $\omega < 1$ and $\tkd{\Delta} < 0$, we have
\begin{equation*}
    0< \eta \leq \frac{1 - \omega}{\kappa} \Longleftrightarrow \eta (1 - \kappa \eta) \tkd{\Delta} \leq \omega \eta \tkd{\Delta}
\end{equation*}
which implies
\begin{equation*}
\Phi(x^k+\eta d^k, (1 - \eta) U^k + \eta f_1(\tilde{y}^{k}))\leq \Phi(x^{k}, U^k) + \eta \omega \tkd{\Delta}.
\end{equation*}
Let us denote by $m$ the smallest integer such that $\delta^{m} \leq (1 - \omega)/\kappa$. Then, the parameter $m_k$ computed during a line search with termination condition \eqref{eq:armijo} satisfies $m_k \leq m$ and, therefore, formula \eqref{eq:updating_complete} implies $\eta_k \geq \delta^{m_k}\geq \delta^m$. Hence, \eqref{eq:eta_min} is satisfied with $\eta_{\min} = \delta^m$.
\end{proof}

Besides the well posedness of the steplength selection rule, Lemma \ref{lemma:linesearch} has further relevant theoretical consequences, stated in the following propositions, whose proof is given in the Appendix \ref{sec:proofs}.

\begin{proposition} \label{prop:H1}
Assume that $f_0,f_1$ satisfy [A1]--[A3] and  that [S1]--[S3] hold for the sequences $\{x^{k}\}_{k \in \N}$, $\{\tilde{y}^{k}\}_{k \in \N}$, $\{\hat{y}^{k}\}_{k \in \N}$, $\{U^k\}_{k\in\N}$. Then for any $k\geq 0$ we have
\begin{equation} \label{eq:Phi_desc}
\Phi(x^{k+1}, U^{k+1}) - \omega \eta_{\min} \tilde{\Delta}_k \leq \Phi(x^{k}, U^k).
\end{equation}
If, in addition, Assumption [A4] is satisfied, we also have
\begin{equation} \label{eq:limits}
0 = \lim_{k \to \infty} \norm{x^{k+1} - x^{k}} = \lim_{k \to \infty} \tilde{\Delta}_k = \lim_{k \to \infty} \norm{\tilde{y}^{k} - x^{k}}= \lim_{k \to \infty} \norm{\hat{y}^{k} - x^{k}}.
\end{equation}
\end{proposition}

\begin{proposition} \label{prop:H2}
Assume that $f_0,f_1$ satisfy [A1]--[A4] and  that [S1]--[S3] hold for the sequences $\{x^{k}\}_{k \in \N}$, $\{\tilde{y}^{k}\}_{k \in \N}$, $\{\hat{y}^{k}\}_{k \in \N}$, $\{U^k\}_{k\in\N}$. Then for any $k\geq 0$ we have
\begin{equation} \label{eq:H2_ineq}
\Phi(x^{k+1}, U^{k+1}) \leq F(\hat{y}^{k}) +\frac 1 2 \rho_k^2 \leq \Phi(x^{k}, U^k) + r_k,
\end{equation}
where 
\begin{eqnarray}
    \rho_k &=& \sqrt{2} \sqrt{-c \tkd{\Delta}}\label{eq:rho_k}\\
    r_k&=&  \rho_k^2 / 2 - d \tkd{\Delta}\label{eq:r_k}
\end{eqnarray}
 with $c,d$ as in Lemma \ref{lem: inequalities}. Moreover, $\lim_{k \to \infty} r_k =  \lim_{k \to \infty} \rho_k=0$.
\end{proposition}
\subsection{Stationarity of the limit points}
The first convergence result is stated in the following theorem.
\begin{theorem}\label{thm:first}
Assume that $f_0,f_1$ satisfy [A1]--[A4] and that [S1]--[S3] hold for the sequences $\{x^{k}\}_{k \in \N}$, $\{\tilde{y}^{k}\}_{k \in \N}$, $\{\hat{y}^{k}\}_{k \in \N}$, $\{U^k\}_{k\in\N}$. 
Then, every limit point $x^*$ of the sequence $\{x^k\}_{k\in\N}$ is stationary for the problem \eqref{eq:problem} and $\lim_{k\to\infty} \Phi(x^k,U^k) = F(x^*)$.
\end{theorem}
\begin{proof}
Let $\{x^{k_j}\}_{j\in\N}$ a subsequence of $\{x^k\}_{k\in\N}$ converging to some $x^*\in\R^n$. Equality \eqref{eq:limits} implies that $x^*$ is a limit point also of the sequences $\{\tilde y^{k_j}\}_{j\in\N}$ and $\{\hat y^{k_j}\}_{j\in\N}$. Moreover, since the sequence $\{\Phi(x^k,U^k)\}_{k\in\N}$ is monotone decreasing and bounded from below by assumption [A4], its limit is finite, i.e., $\lim_{k\to\infty}\Phi(x^k,U^k) = \Phi^*\in\R$. Moreover, we have
\begin{eqnarray*}
    \Phi^* &=& \lim_{k\to\infty} \Phi(x^k,U^k) =   \lim_{k\to\infty} f_0(x^k) + U^k\\
    &\geq & \lim_{k\to \infty} F(x^k) =\lim_{j\to\infty} F(x^{k_j}) = f_0(x^*) + \lim_{j\to\infty} f_1(x^{k_j})\\ &\geq& f_0(x^*) + f_1(x^*) = F(x^*),
\end{eqnarray*}
where the first inequality follows from Proposition \ref{prop:Uk} and the last one from the semicontinuity of $f_1$.
Therefore we can write
 \begin{equation*}
     \Phi^*-f_0(x^*) \geq f_1(x^*).
 \end{equation*}
On the other side, Proposition \ref{prop:H2} implies that 
\begin{equation}\label{eq:stat1}
    \lim_{k\to\infty} F(\hat y^k) = \lim_{k\to\infty}\Phi(x^k,U^k) = \Phi^*.
\end{equation}
Let $\hat w^k\in \partial f_1(\hat y^k)$ be defined as in Lemma \ref{lem:subgrad_bound}. We have
   \begin{equation*}
       \lim_{j\to\infty} w^{k_j} = \lim_{j\to\infty} -\nabla f_0(\hat y^{k_j}) = -\nabla f_0(x^*),
   \end{equation*}
where we exploited  the continuity of $\nabla f_0$ and the convergence of $\{\hat y^{k_j}\}_{j\in\N}$ to $x^*$. On the other side, the subgradient $w^{k_j}$ satisfies the following inequality
   \begin{eqnarray*}
       f_1(x^*)&\geq& f_1(\hat y^{k_j}) +\langle w^{k_j}, x^* - \hat y^{k_j}\rangle\\
       &=& F(\hat y^{k_j}) +\langle w^{k_j}, x^* - \hat y^{k_j} \rangle - f_0(\hat y^{k_j}),
   \end{eqnarray*}
where the last equality is obtained by adding and subtracting  $f_0(\hat y^{k_j}) $ to the right-hand-side. Taking the limits for $j\to\infty$ on both sides, in view of \eqref{eq:stat1}, we obtain
\begin{equation*}
     \Phi^*-f_0(x^*) \leq f_1(x^*).
 \end{equation*}
Hence, we have $\Phi^* = F(x^*)$.
Finally, we observe that, from Lemma \eqref{lem:subgrad_bound}, there exists $\hat v^k\in\partial F(\hat y^k)$ such that $\lim_k\|\hat v^k\| = 0$. Therefore, we can apply Lemma 2.1 in \cite{Frankel-etal-2015} to conclude that $x^*$ is stationary.  
\end{proof}

\subsection{Convergence analysis in the Kurdyka--{\L}ojasiewicz framework}

In this section we give more insight on the convergence of the sequences associated to conditions [S1]--[S3] adopting in particular the framework of the KL inequality, which is stated in the following definition.
\begin{definition}\label{def:KL} \cite[Definition 3]{Bolte-etal-2014}
    Let $f:\R^n\rightarrow \R\cup\{+\infty\}$ be proper and lower semicontinuous. The function $f$ satisfies the KL inequality at the point $x^*\in\dom(\partial f)$ if there exist $\nu>0$, a neighbourhood $U$ of $x^*$, and a continuous concave function $\xi:[0,\nu)\rightarrow [0,+\infty) $ such that $\xi(0) = 0$, $\xi$ is $C^1$ on $(0,\nu)$, $\xi'(s)>0$ for all $s\in(0,\nu)$, and the following inequality holds
    \begin{equation}\nonumber
        \xi'(f(x)-f(x^*))\mathrm{dist}(0,\partial f(x))\geq 1,
    \end{equation}
    for all $x\in U \cap \{y\in\R^n: \ f(x^*)< f(y) < f(x^*)+\nu\}$. If $f$ satisfies the KL inequality for all $x^*\in\dom(\partial f)$, then $f$ is called a KL function.
\end{definition}
The analysis performed in this section runs under the following assumption:
\begin{description}
    \item[{[A5]}] The function $\mathcal{F} : \R^n \times \R \to \R\cup\{+\infty\}$ defined as
\begin{equation} \label{eq:merit_func}
\mathcal{F}(x, \rho) = F(x) + \frac{1}{2} \rho^2
\end{equation}
is a KL function.
\end{description}

Assumption [A5] is satisfied by a wide class of functions, including real-analytic and semi-algebraic mappings, and functions definable in an o-minimal structure. In particular, if $f_0$ and $f_1$ fall both into one of the aforementioned classes, then $\mathcal{F}$ belongs to the same class and thus satisfies the KL property. Notice that the function $f_1$ defined in \eqref{eq:f1} is real-analytic \cite[Corollary 1(\emph{v})]{Hurault-etal-2024}. This assumption is crucial to ensure the convergence of the algorithm to a stationary point, as stated in the following theorem. 
%
%

\begin{theorem} \label{thm:main}
Assume that $f_0,f_1$ satisfy [A1]--[A5] and that [S1]--[S3] hold for the sequences $\{x^{k}\}_{k \in \N}$, $\{\tilde{y}^{k}\}_{k \in \N}$, $\{\hat{y}^{k}\}_{k \in \N}$, $\{U^k\}_{k\in\N}$. Assume, in addition, that the sequence $\{x^{k}\}_{k \in \N}$ is bounded. Then, $\{x^{k}\}_{k \in \N}$ converges to a stationary point of $F$.
\end{theorem}

The boundedness of the sequence $\{x^{k}\}_{k \in \N}$ is a common assumption in convergence analysis of nonconvex optimization algorithms. The assumption is naturally satisfied if the objective function $F$ is coercive, i.e., $F(x) \to +\infty$ as $\norm{x} \to +\infty$.

The proof of Theorem \ref{thm:main} is given by showing that [S1]--[S3] are a special case of the abstract scheme defined in \cite[Theorem 5(iii)]{Bonettini-etal-2023}, which is restated below.

\begin{theorem}\label{thm:abstract}
Let $\mathcal{F} : \R^n \times \R^m \to \overline{\R}$ be a proper, lower semicontinuous KL function. Consider any sequence $\{(x^{k}, \rho_{k})\}_{k \in \N} \subset \R^n \times \R^m$ and assume that there exist a proper, lower semicontinuous, bounded from below function $\Phi : \R^n \times \R^q \to \overline{\R}$ and four sequences $\{u^{k}\}_{k \in \N} \subset \R^n$, $\{s^{k}\}_{k \in \N} \subset \R^q$, $\{\rho_{k}\}_{k \in \N} \subset \R^m$, $\{d_k\}_{k \in \N} \subset \R_{\geq 0}$ such that the following relations are satisfied.
\begin{description}
    \item[{[H1]}] There exists a positive real number $a$ such that
    \[ \Phi(x^{k+1}, s^{k+1}) + a d_k^2 \leq \Phi(x^{k}, s^{k}), \quad \forall k \geq 0. \]
    \item[{[H2]}] There exists a sequence of non--negative real numbers $\{r_k\}_{k \in \N}$ with $\lim_{k \to \infty} r_k = 0$ such that
    \[ \Phi(x^{k+1}, s^{k+1}) \leq \mathcal{F}(u^{k}, \rho_{k}) \leq \Phi(x^{k}, s^{k}) + r_k, \quad \forall k \geq 0. \]
    \item[{[H3]}] There exists a subgradient $w^{k} \in \partial \mathcal{F}(u^{k}, \rho_{k})$ such that
    \[ \norm{w^{k}} \leq b \sum_{i \in \mathcal{I}} \theta_i d_{k+1-i}, \quad \forall k \geq 0, \]
    where $b$ is a positive real number, $\mathcal{I} \subset \mathbb{Z}$ is a nonempty, finite index set and $\theta_i \geq 0, i \in \mathcal{I}$ with $\sum_{i \in \mathcal{I}} \theta_i = 1$ ($d_j = 0$ for $j \leq 0$).
    \item[{[H4]}] If $\{(x^{k_j}, \rho_{k_j})\}_{j \in \N}$ is a subsequence of $\{(x^{k}, \rho_{k})\}_{k \in \N}$ converging to some $(x^*, \rho^*) \in \R^n \times \R^m$, then we have
    \[ \lim_{j \to \infty} \norm{u^{k_j} - x^{k_j}} = 0, \quad \lim_{j \to \infty} \mathcal{F}(u^{k_j}, \rho_{k_j}) = \mathcal{F}(x^*, \rho^*). \]
    \item[{[H5]}] There exists a positive real number $p > 0$ such that
    \[ \norm{x^{k+1} - x^{k}} \leq p d_k, \quad \forall k \geq 0. \]
\end{description}
Moreover, assume that $\{(x^{k}, \rho_{k})\}_{k \in \N}$ is bounded and $\{\rho_{k}\}_{k \in \N}$ converges. Then, the sequence $\{(x^{k}, \rho_{k})\}_{k \in \N}$ converges to a stationary point of $\mathcal{F}$.
\end{theorem}

We will show that the sequences associated to [S1]--[S2], with a proper setting of the surrogate function $\Phi$ and all the auxiliary sequences, satisfies the assumptions [H1]--[H5], following the same convergence analysis adopted in \cite{Bonettini-etal-2024}. Then, we will conclude the proof of Theorem \ref{thm:main} by applying the aforementioned abstract convergence result. \\




\begin{proof}[Proof of Theorem \ref{thm:main}]
Set the merit functions $\Phi$, $\mathcal F$ as in \eqref{eq:Phi_def} and \eqref{eq:merit_func}, respectively.
From Proposition \ref{prop:H1} we directly obtain [H1] with $a = \omega \eta_{\min}$, $d_k = \sqrt{-\tilde\Delta_k}$ and $s^k = U^k$.

Let us set $u^k = \hat y^k$ and $\rho_k, r_k$ as in Proposition \ref{prop:H2}. In these settings, [H2] is satisfied. 
As for [H3], let $\hat v^k\in\partial F(\hat y^k)$ and $q>0$ be such that inequality \eqref{eq:subgrad_ineq} is satisfied. Setting $w^{k} = ( \hat{v}^{k} , \rho_k)^T$, since $\mathcal{F}$ is separable and, therefore, $\partial \mathcal{F}(x,\rho) = \partial F (x) \times \{ \rho \}$, we have $w^{k} \in \partial \mathcal{F}(\hat{y}^{k}, \rho_k)$. Hence, by the triangular inequality and in view also of \eqref{eq:rho_k}, we obtain
\begin{equation} \nonumber
\norm{w^{k}} \leq \norm{\hat{v}^{k}} + |\rho_k| \leq q \sqrt{-\tkd{\Delta}} + \sqrt{2} \sqrt{-c \tkd{\Delta}}.
\end{equation}
Therefore, [H3] holds with $b = q + \sqrt{2c}$,  $\mathcal{I} = \{1\}$, $\theta_1 = 1$.

Furthermore, from the proof of Theorem \ref{thm:first}, we have that $\lim_{k\to\infty}\|\hat y^k -x^k\| = 0$ and $\lim_{k\to\infty} F(\hat y^k) = \lim_{k\to\infty} F(x^k) = F(x^*)$ (see formula \eqref{eq:stat1}), for every limit point $x^*$ of $\{x^k\}_{k\in\N}$. Moreover, we have $\rho^*:=\lim_{k\to\infty} \rho_k  = 0$. Therefore, [H4] is satisfied.

To conclude, since $\norm{x^{k+1} - x^k} \leq \norm{\tk{y} - x^k}$, inequality \eqref{eq:3epsilon} entails [H5] with  $p = \sqrt{\frac{4}{\theta} (1+\mu) }$.
Then, Theorem \ref{thm:abstract} applies and guarantees that the sequence ${(x^k,\rho_k)}_{k \in \N}$ converges to a stationary point $(x^*,\rho^*)$ of $\mathcal{F}$. Notice that, since $\mathcal F$ is separable, $(x^*,\rho^*)$ is stationary for $\mathcal{F}$ if and only if $\rho^* = 0$ and $0 \in \partial F (x^*)$. Hence $x^*$ is a stationary point for $F$ and $\{x^k\}_{k \in \N}$ converges to it.
\end{proof}

\section{Numerical experiments}
\label{sec:numerical_experiments}

In this section, we evaluate the empirical performance of the proposed PnP-IPA algorithm. We focus on the deblurring inverse problem, using the 8 real-world camera shake kernels from \cite{Levin-etal-2009} as done in previous works \cite{Pesquet-etal-2021,Zhang-etal-2021}, and benchmark our approach against state-of-the-art PnP methods. Results are obtained on a subset of the first 10 $256 \times 256$ center-cropped color images of the CBSD68 dataset \cite{Martin-etal-2001}, hereafter referred to as CBSD10. In particular, our experiments are designed to demonstrate three aspects of our method:
\begin{enumerate}
    \item the effectiveness of PnP-IPA in a high Gaussian noise regime;
    \item the superiority of our algorithm over previous PnP methods in a low Gaussian noise regime, where we can leverage the flexibility of the line--search to optimally tune the regularization parameter;
    \item the convergence of PnP-IPA in a deblurring scenario with Cauchy noise, which is a heavy-tailed distribution which leads to a nonconvex data-fidelity term.
\end{enumerate}
With regards to the Gaussian noise, for which the data-fidelity term is convex, we compare the method to three aforementioned state-of-the-art PnP algorithms with convergence guarantees: Prox-PnP \cite{Hurault-etal-2022a}, relaxed-Prox-PnP and alpha-Prox-PnP \cite{Hurault-etal-2024}. Moreover, we also compare with DPIR \cite{Zhang-etal-2021}, which is a widely used PnP method based on the HQS algorithm and a DRUNet denoiser that achieves excellent visual results but lacks convergence guarantees. For the Cauchy noise, we compare with the prox-PnP, for which convergence is guaranteed even in the nonconvex setting (unlike alpha-Prox-PnP), and with a Gradient Descent method plus backtracking using the Gradient-Step prior in a RED framework, as described in \cite{Cohen-etal-2021}.
 
\subsection{Implementation details}  \label{sec:implementation_details}

Our experiments\footnote{The code to reproduce the numerical experiments is available at \url{https://github.com/crisparenti/PnP-IPA}.} are implemented in Python using the DeepInverse \cite{Tachella-etal-2025} library, which provides a modular framework for implementing algorithms for inverse problems. For a fair comparison, the Gradient-Step denoiser and the DRUNet used in the DPIR algorithm are implemented using the neural network architectures and pre-trained weights given by the library. Regarding the Prox-PnP algorithm, its relaxed version and the alpha-Prox-PnP (as well as PnP-IPA), we use the official weights provided by the authors of \cite{Hurault-etal-2022a}, which can be downloaded from their GitHub repository. The parameters of the high-noise Gaussian setting are set accordingly to the guidelines provided in the respective references, while for the low-noise Gaussian setting and for the Cauchy noise an extensive grid search has been performed for all the methods to ensure the best possible PSNR. The hyperparameters configuration for the baseline methods is given in Appendix \ref{sec:baselines_hyperparameters}. All methods are stopped when the relative error $\norm{ x^k - x^{k-1} } / \norm{ x^{k-1} }$ falls below a threshold of $10^{-4}$. Moreover, we fixed the seed to 35. All experiments ran on a machine equipped with an Intel Core i7-12700H CPU, 16 GB of RAM, and an NVIDIA GeForce RTX 4050 (Laptop) GPU with 6 GB of VRAM.

Regarding our proposed PnP-IPA algorithm, the line-search hyperparameters are set to $\delta = 0.5$ and $\omega = 10^{-4}$. The step--size $\alpha_k$, which we remind that it can be freely chosen without any theoretical constraint (apart from lower and upper bounds), is initially set to $10^{6}$ and reduced by a factor of 3 every $N_{\alpha}$ iterations until $\alpha_k \leq 10^{2}$, and then it remains constant. The rationale behind this choice is given in Section \ref{sec:IPA_remarks}. We set $N_{\alpha} = 10$ for the Gaussian-noise experiments and $N_{\alpha} = 25$ for the Cauchy-noise experiments, where a slower decrease of $\alpha_k$ was found to be beneficial. The regularization parameter $\lambda$ is tuned for each noise level and each method to achieve the best possible PSNR. 

Our results have been obtained using the BenchOpt framework \cite{Moreau-etal-2022}, which ensures a standardized evaluation protocol and fair comparison across different algorithms. In particular, the reported computation time is handled directly by the Benchopt library and corresponds to the wall-clock from solver start until the stopping criterion is met.

\begin{remark}
    In \cite{Hurault-etal-2022a, Hurault-etal-2024}, the regularization parameter is multiplied to the data fidelity term, leading to an objective functional written as
    \begin{equation} \label{eq:proxPnP_objective}
        \min_x \lambda f_{data}(x) + \phisigma(x).
    \end{equation}
    Hence it differs from our setting, where the regularization parameter is multiplied to the regularization term. However, the two formulations are equivalent and can be easily transformed one into the other by a simple change of variable. In particular, if we denote by $\lambda'$ the regularization parameter in our setting, then the corresponding regularization parameter in the formulation of \cite{Hurault-etal-2022a, Hurault-etal-2024} is given by $\lambda = 1 / \lambda'$. Then, tuning $\lambda'$ is equivalent to tuning $\lambda$. In the following, to make the comparison easier to understand, we will refer to the latter formulation \ref{eq:proxPnP_objective}.
\end{remark}

\subsection{Image deblurring under high Gaussian noise}

We first consider the standard image deblurring task under additive white Gaussian noise with standard deviation $\nu = 0.05$, which is a high level of noise. We remind, as a classical result, that the data-fidelity term related to Gaussian noise is given by the standard least--squares function $f_{data}(x) = \frac{1}{2} \norm{Ax - y}^2$, which is convex. The Lipschitz constant of its gradient is $L_{f_{data}} =  \norm{A^T A}$, equal to 1 for normalized kernels (as the ones we test the methods on). Regarding PnP-IPA, we set the regularization parameter $\lambda$ to $3$ and the denoising level $\sigma$ to $0.05 = 1 \nu$.

In this aggressive noise regime a high level of regularization is required to stabilize the optimization process; hence, the constraints on the regularization parameter imposed by Prox-PnP do not limit the performance of the algorithm, and all the methods perform well. In particular, as shown in Table \ref{tab:high_noise_table}, Prox-PnP, alpha-Prox-PnP and PnP-IPA achieve comparable PSNR values, while DPIR performs better (but without theoretical guarantees). Our method, nevertheless, is faster than Prox-PnP and its variants.

Lastly, we remark that relaxed-Prox-PnP is not included in the comparison since, in this high noise regime, the optimal value of the relaxation parameter provided in the reference paper is exactly 1, which makes the algorithm equivalent to Prox-PnP.

\begin{table}[h!]
\begin{center}
\caption{High Gaussian Noise deblurring: average PSNR (dB) and computation time (s) evaluated across 10 images and 8 blur kernels.}
\begin{tabular}{lcc}
\toprule
Method & Average PSNR (dB) & Average Time (s) \\
\midrule
DPIR & 27.78 & 0.29 \\
PnP-IPA & 27.53 & 7.47 \\
Prox-PnP & 27.50 & 10.09 \\
alpha-PnP & 27.50 & 10.05 \\
\bottomrule
\end{tabular}
\label{tab:high_noise_table}
\end{center}
\end{table}

In Figure \ref{fig:high_noise_figure}, we plot the evolution of PSNR and relative error along time, as well as the reconstructed images from an example image provided in the CBSD68 dataset.

\begin{figure}[h!]
    \centering

    \begin{subfigure}{0.48\textwidth}
        \centering
        \includegraphics[width=\linewidth]{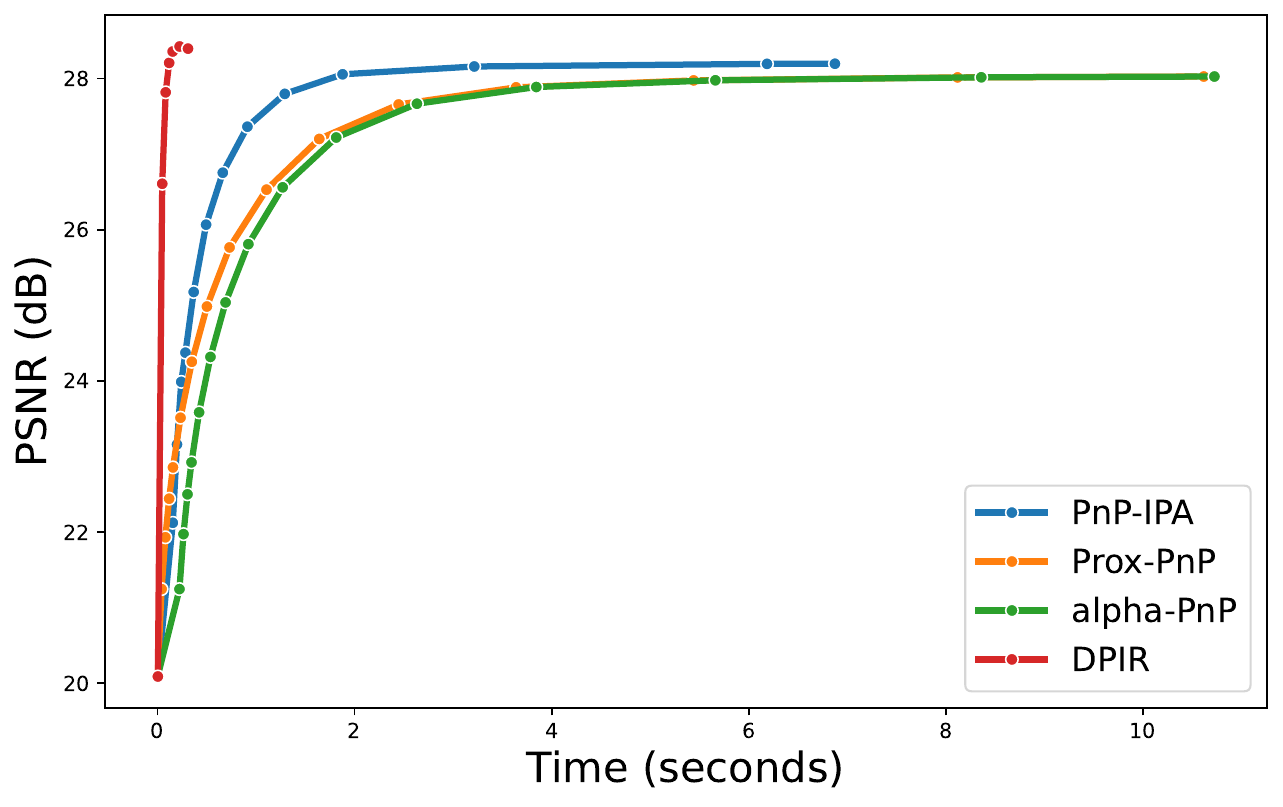}
        \caption{PSNR evolution along time.}
    \end{subfigure}\hfill
    \begin{subfigure}{0.48\textwidth}
        \centering
        \includegraphics[width=\linewidth]{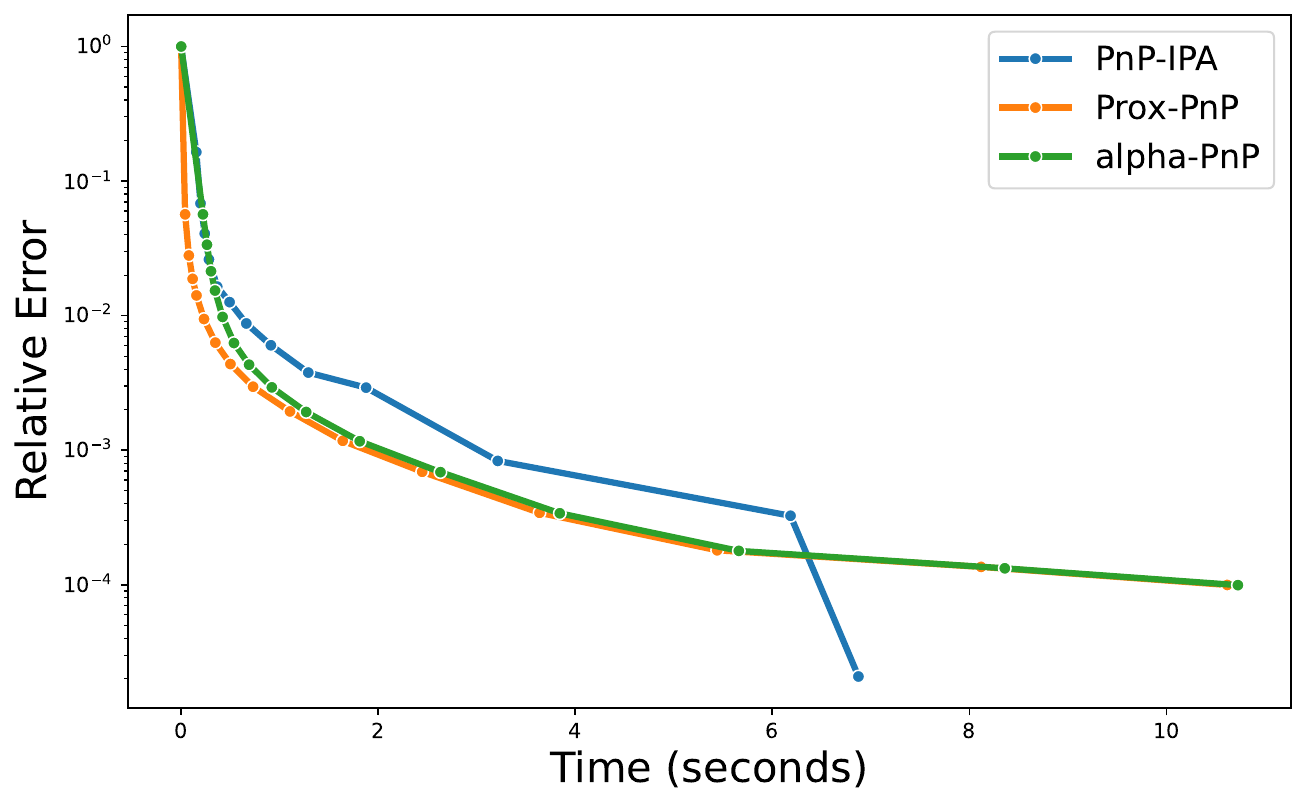}
        \caption{Relative error evolution along time.}
    \end{subfigure}
    
    \vspace{0.5cm}

    {\renewcommand{\spysize}{1.67cm}%
    \begin{subfigure}{0.32\textwidth}
        \centering
        \spyimg{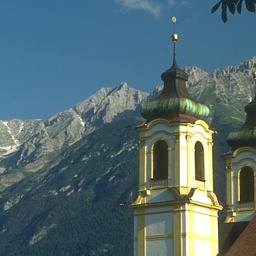}{3.48}{3.06}{1.085}{0.250}
        \caption{Clean image}
    \end{subfigure}\hfill
    \begin{subfigure}{0.32\textwidth}
        \centering
        \spyimg{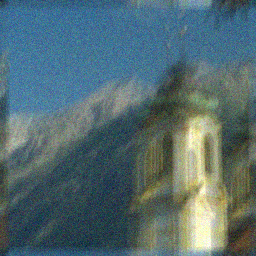}{3.48}{3.06}{1.085}{0.250}
        \caption{Blurred image}
    \end{subfigure} \hfill}
     \begin{subfigure}{0.32\textwidth}
        \centering
        \includegraphics[width=\linewidth]{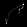}
         \vspace{0.17cm}
        \caption{Blur kernel ($4^{\text{th}}$ kernel of \cite{Levin-etal-2009})}
    \end{subfigure}\hfill
    
    \vspace{0.5cm} 
    
    \renewcommand{\spysize}{1.2cm}%
    \begin{subfigure}{0.23\textwidth}
        \centering
        \spyimg{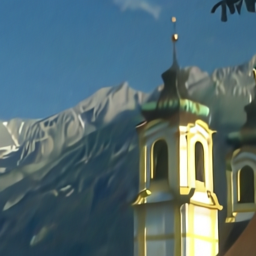}{2.5}{2.2}{0.78}{0.18}
        \caption{Reconstruction with PnP-IPA}
    \end{subfigure}\hfill
    \begin{subfigure}{0.23\textwidth}
        \centering
        \spyimg{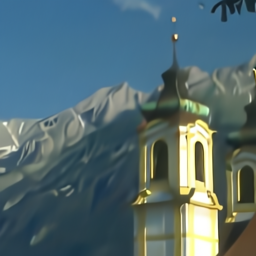}{2.5}{2.2}{0.78}{0.18}
        \caption{Reconstruction with Prox-PnP}
    \end{subfigure} \hfill
    \begin{subfigure}{0.23\textwidth}
        \centering
        \spyimg{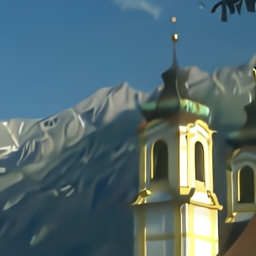}{2.5}{2.2}{0.78}{0.18}
        \caption{Reconstruction with alpha-Prox-PnP}
    \end{subfigure}\hfill
    \begin{subfigure}{0.23\textwidth}
        \centering
        \spyimg{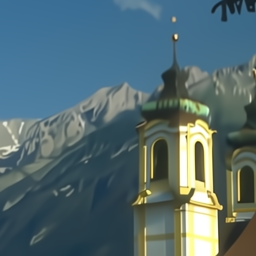}{2.5}{2.2}{0.78}{0.18}
        \caption{Reconstruction with DPIR}
    \end{subfigure}\hfill
    
    \caption{Reconstruction results and convergence comparison under aggressive Gaussian noise ($\nu = 0.05$). In this high noise regime, all methods perform well, with DPIR achieving the best PSNR but without theoretical guarantees. PnP-IPA achieves comparable PSNR to Prox-PnP and its variants, while being slightly faster. The image is taken from the CBSD68 dataset \cite{Martin-etal-2001}.}
    \label{fig:high_noise_figure}
\end{figure}

\subsection{Image deblurring under weak Gaussian noise}

We compare the same methods under weak Gaussian noise. Since in this setting we can trust more the data fidelity term, the optimal regularization parameter $\lambda$ can be higher than in the aggressive noise case. Then, the constraints on $\lambda$ imposed by Prox-PnP and its variants become relevant, as they force the algorithm to use a suboptimal $\lambda$ that is too small for this noise level. In contrast, PnP-IPA is completely agnostic to constraints. This allows us to apply less regularization (tuning $\lambda$ to its true optimal value for image quality) without compromising convergence.

PnP-IPA is run with $\lambda = 10$, which largely exceeds the upper bound on $\lambda$ imposed by Prox-PnP, and the denoising level $\sigma$ is set to $0.02 = 2 \nu$. Quantitative results are reported in Table~\ref{tab:low_noise_table}. We observe that PnP-IPA consistently outperforms Prox-PnP and its variants. In Figure \ref{fig:low_noise_figure} we show a visual comparison of the restored images.

\begin{table}[h!]
\begin{center}
\caption{Low Gaussian Noise deblurring: average PSNR (dB) and computation time (s) evaluated across 10 images and 8 blur kernels.}
\begin{tabular}{lcc}
\toprule
Method & Average PSNR (dB) & Average Time (s) \\
\midrule
DPIR & 33.54 & 0.30 \\
PnP-IPA & 33.10 & 9.81 \\
alpha-PnP & 33.02 & 6.36 \\
Prox-PnP-Relaxed & 32.61 & 7.14 \\
Prox-PnP & 32.60 & 8.74 \\
\bottomrule
\end{tabular}
\label{tab:low_noise_table}
\end{center}
\end{table}

\begin{figure}[h!]
    \centering
    
    \begin{subfigure}{0.48\textwidth}
        \centering
        \includegraphics[width=\linewidth]{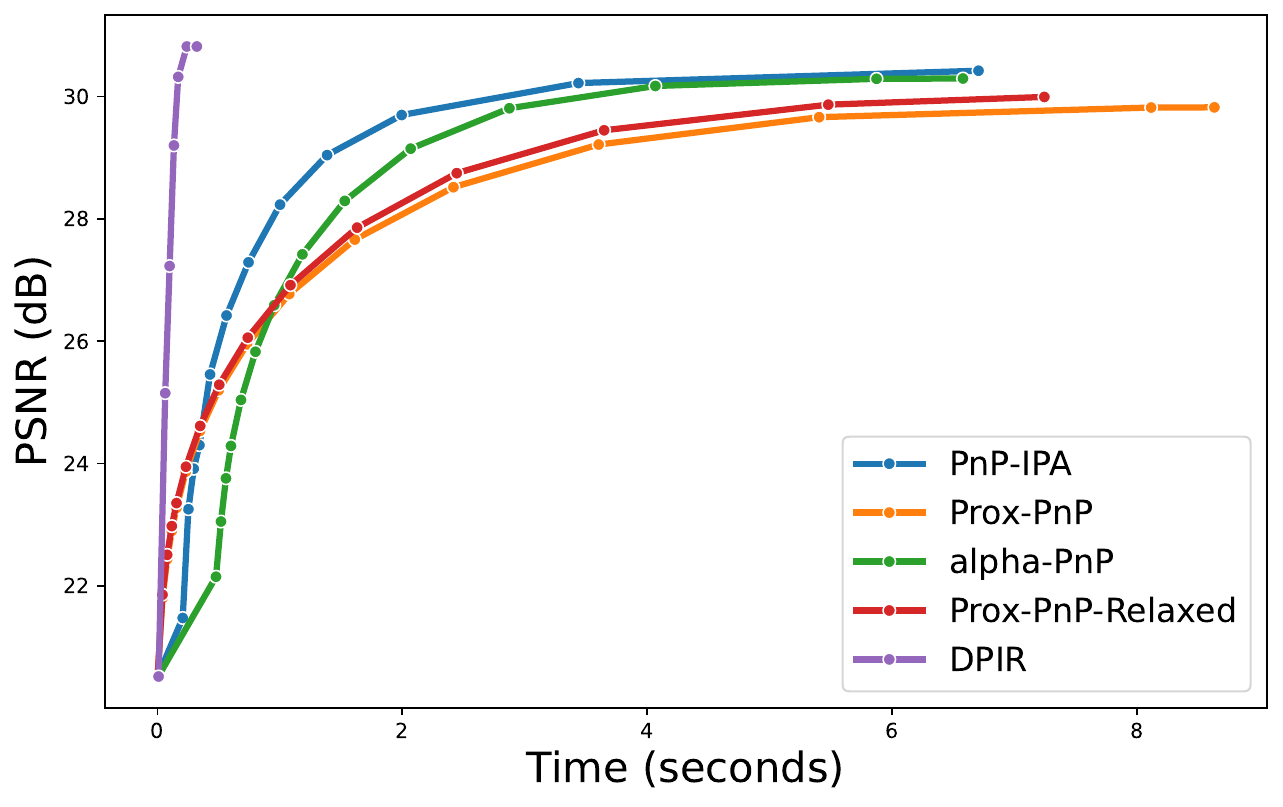}
        \caption{PSNR evolution along time.}
    \end{subfigure}\hfill
    \begin{subfigure}{0.48\textwidth}
        \centering
        \includegraphics[width=\linewidth]{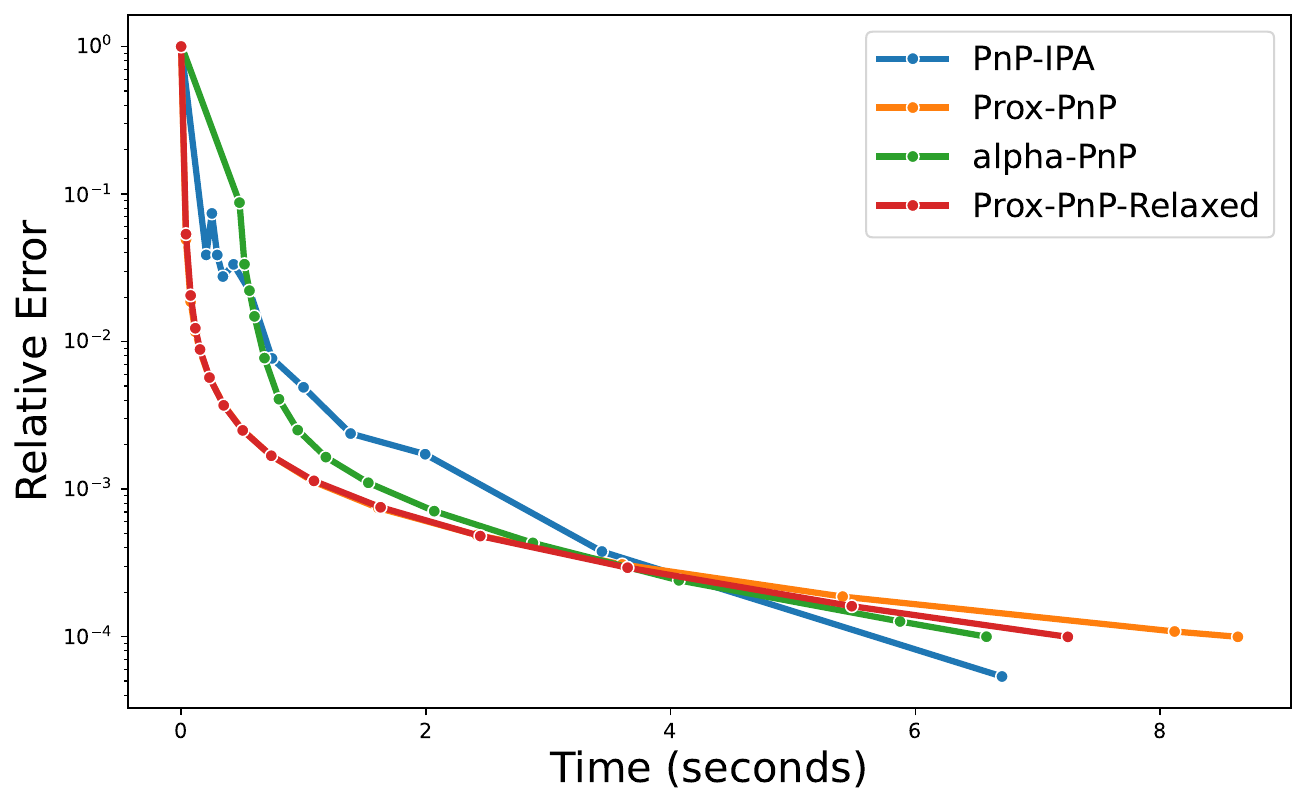}
        \caption{Relative error evolution along time.}
    \end{subfigure}
    
    \vspace{0.5cm}

    {\renewcommand{\spysize}{1.78cm}%
    \begin{subfigure}{0.32\textwidth}
        \centering
        \spyimg{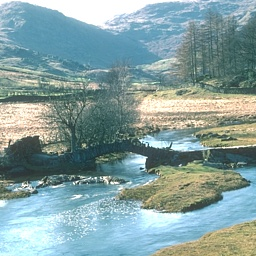}{2.40}{2.72}{1.387}{0.320}
        \caption{Clean image}
    \end{subfigure}\hfill
    \begin{subfigure}{0.32\textwidth}
        \centering
        \spyimg{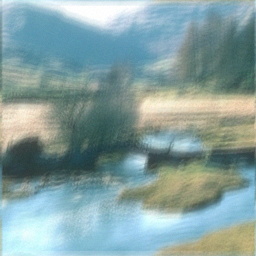}{2.40}{2.72}{1.387}{0.320}
        \caption{Blurred image}
    \end{subfigure} \hfill}
     \begin{subfigure}{0.32\textwidth}
        \centering
        \includegraphics[width=\linewidth]{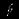}
         \vspace{0.17cm}
        \caption{Blur kernel ($1^{\text{st}}$ kernel of \cite{Levin-etal-2009})}
    \end{subfigure}\hfill
    
    \vspace{0.5cm}
    
    \renewcommand{\spysize}{1.0cm}%
    \begin{subfigure}{0.18\textwidth}
        \centering
        \spyimg{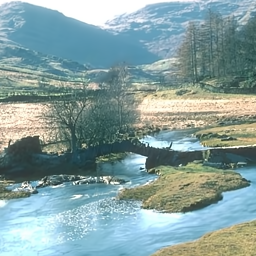}{1.35}{1.53}{0.78}{0.18}
        \caption{Reconstruction with DPIR}
    \end{subfigure}\hfill
    \begin{subfigure}{0.18\textwidth}
        \centering
        \spyimg{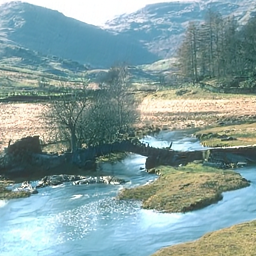}{1.35}{1.53}{0.78}{0.18}
        \caption{Reconstruction with PnP-IPA}
    \end{subfigure}\hfill
    \begin{subfigure}{0.18\textwidth}
        \centering
        \spyimg{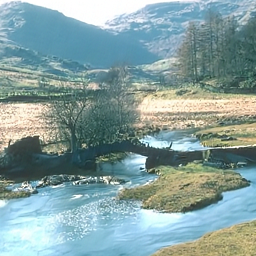}{1.35}{1.53}{0.78}{0.18}
        \caption{Reconstruction with alpha-PnP}
    \end{subfigure}\hfill
    \begin{subfigure}{0.18\textwidth}
        \centering
        \spyimg{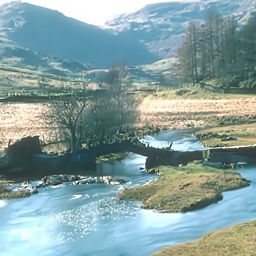}{1.35}{1.53}{0.78}{0.18}
        \caption{Reconstruction with relaxed-PnP}
    \end{subfigure} \hfill
    \begin{subfigure}{0.18\textwidth}
        \centering
        \spyimg{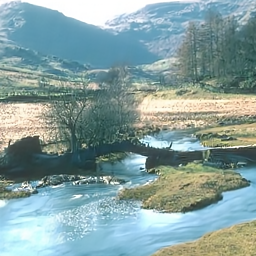}{1.35}{1.53}{0.78}{0.18}
        \caption{Reconstruction with Prox-PnP}
    \end{subfigure} \hfill

    \caption{Reconstruction results and convergence comparison under low Gaussian noise ($\nu = 0.01$). In this low noise regime, PnP-IPA outperforms Prox-PnP and its relaxed version by a significant margin, thanks to its ability to use the optimal regularization parameter without any constraint. DPIR achieves the best PSNR but without theoretical guarantees. The image is taken from the CBSD68 dataset \cite{Martin-etal-2001}.}
    \label{fig:low_noise_figure}
\end{figure}

\subsection{Deblurring under Cauchy noise}

We move beyond convex data fidelity and evaluate the algorithms on a nonconvex task: deblurring under Cauchy noise. Heavy-tailed Cauchy noise is used to model impulsive noise or outliers, and finds applications in SAR imaging \cite{Bhuiyan-etal-2007}, radar signal processing \cite{Tsihrintzis-Nikias-1997}, underwater acoustic signal processing \cite{Idan-Speyer-2010}, as well as many other domains where extreme deviations from the mean are common. It is parametrized by a scale parameter $\gamma$ that controls the tail heaviness: higher $\gamma$ corresponds to heavier tails and more extreme outliers. Its probability density function is given by
\begin{equation}
    p(x) =  \frac{1}{\pi} \frac{\gamma}{\gamma^2 + x^2}.
\end{equation}
In practice, the noisy image is clipped to [0,1] before reconstruction, a common practice for displayable images. The corresponding Maximum A Posteriori (MAP) estimation yields, up to additive and positive multiplicative constants that do not affect the minimization, the following data-fidelity term, which is nonconvex:
\begin{equation}
    f_{data}(x) = \frac{1}{2} \sum_{i=1}^m \log\left( \gamma^2 + (Ax - y)_i^2 \right).
    \label{eq:cauchy_data_fidelity}
\end{equation}
Furthermore, the Lipschitz constant of its gradient is exactly $L_{f_{data}} = \norm{A^T A} / \gamma^2$, which is extremely large for small $\gamma$ (i.e., heavy-tailed noise). For instance, with $\gamma = 0.01$, we have $L_{f_{data}} = 10^4$ for normalized kernels.

Under these conditions, the theoretical bounds for Prox-PnP and its relaxed version become extremely restrictive, since the parameter $\lambda$ should be set equal to $1.49 \cdot 10^{-4}$ to satisfy $\lambda L_{f_{data}} < 1.5$. Moreover, alpha-Prox-PnP cannot be applied, since it requires the data fidelity to be convex and the relaxed version of Prox-PnP does not significantly relax the constraints on $\lambda$, which would be set in the same order of magnitude. Furthermore, DPIR is omitted from this scenario: despite an extensive hyperparameters tuning, we observed that the algorithm consistently diverges without producing any viable intermediate result. In contrast, PnP-IPA converges without any constraint. 

Prox-PnP runs with the largest admissible proximal--gradient step--size $\tau = 1.49 / L_{f_{data}} = 1.49 \cdot 10^{-4}$ and $\sigma = 0.006$. For the Gradient Descent method, we set the initial step--size to $10^{-3}$ and let it be handled by the backtracking procedure, while the regularization parameter is set to $\lambda = 10^{-3}$ and the denoising level to $\sigma = 0.015$. The initial step size is divided by a factor $\beta =1.5$ until the descent condition is satisfied. For PnP-IPA, we set $\lambda = 1 / 300$ and $\sigma = 0.03$. These parameters have been tuned to achieve the best possible PSNR for each method after an extensive grid search.

Intuitively, the heavy tails of the Cauchy distribution result in a data-fidelity term that, locally, is steep enough to make the Lipschitz constant of its gradient extremely large; convergence of Prox-PnP therefore requires a tiny $\lambda$ (i.e., a strong regularization weight $1/\lambda$ in our splitting convention of equation \ref{eq:proxPnP_objective}). PnP-IPA, having no such constraint, can afford a regularization weight roughly one order of magnitude smaller, granting more credit to the data-fidelity term and yielding sharper reconstructions.

Table \ref{tab:cauchy_table} reports the numerical results. PnP-IPA outperforms the other methods, both in terms of PSNR and computational time. In Figure \ref{fig:cauchy_noise_figure} we report a visual comparison of the restored images, as well as the evolution of PSNR and residual along iterations.

\begin{table}[h!]
\begin{center}
\caption{Cauchy Noise deblurring: average PSNR (dB) and computation time (s) evaluated across 10 images and 8 blur kernels.}
\begin{tabular}{lcc}
\toprule
Method & Average PSNR (dB) & Average Time (s) \\
\midrule
PnP-IPA & 30.77 & 8.48 \\
GS-GD-Backtracking & 30.49 & 24.12 \\
Prox-PnP & 29.15 & 21.97 \\
\bottomrule
\end{tabular}
\label{tab:cauchy_table}
\end{center}
\end{table}

\begin{figure}[h!]
    \centering
    
    \begin{subfigure}{0.48\textwidth}
        \centering
        \includegraphics[width=\linewidth]{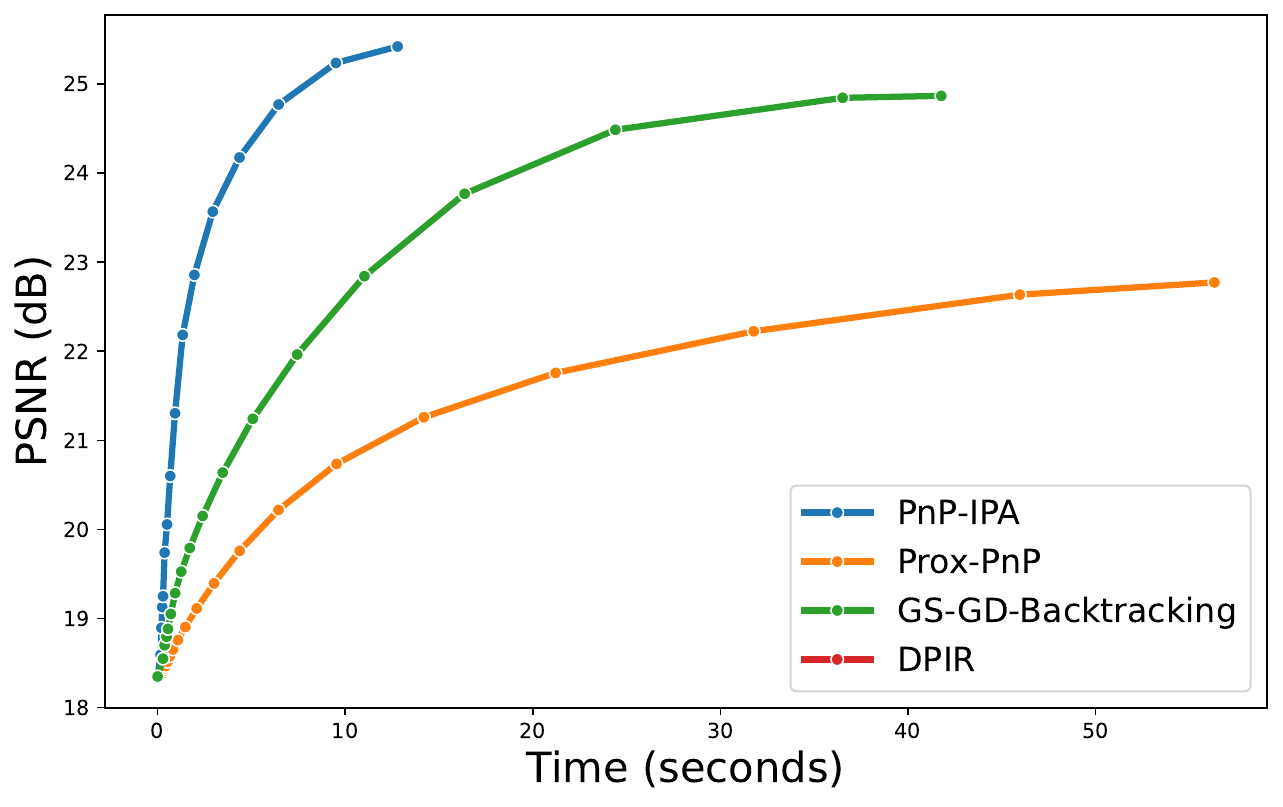}
        \caption{PSNR evolution along time.}
    \end{subfigure}\hfill
    \begin{subfigure}{0.48\textwidth}
        \centering
        \includegraphics[width=\linewidth]{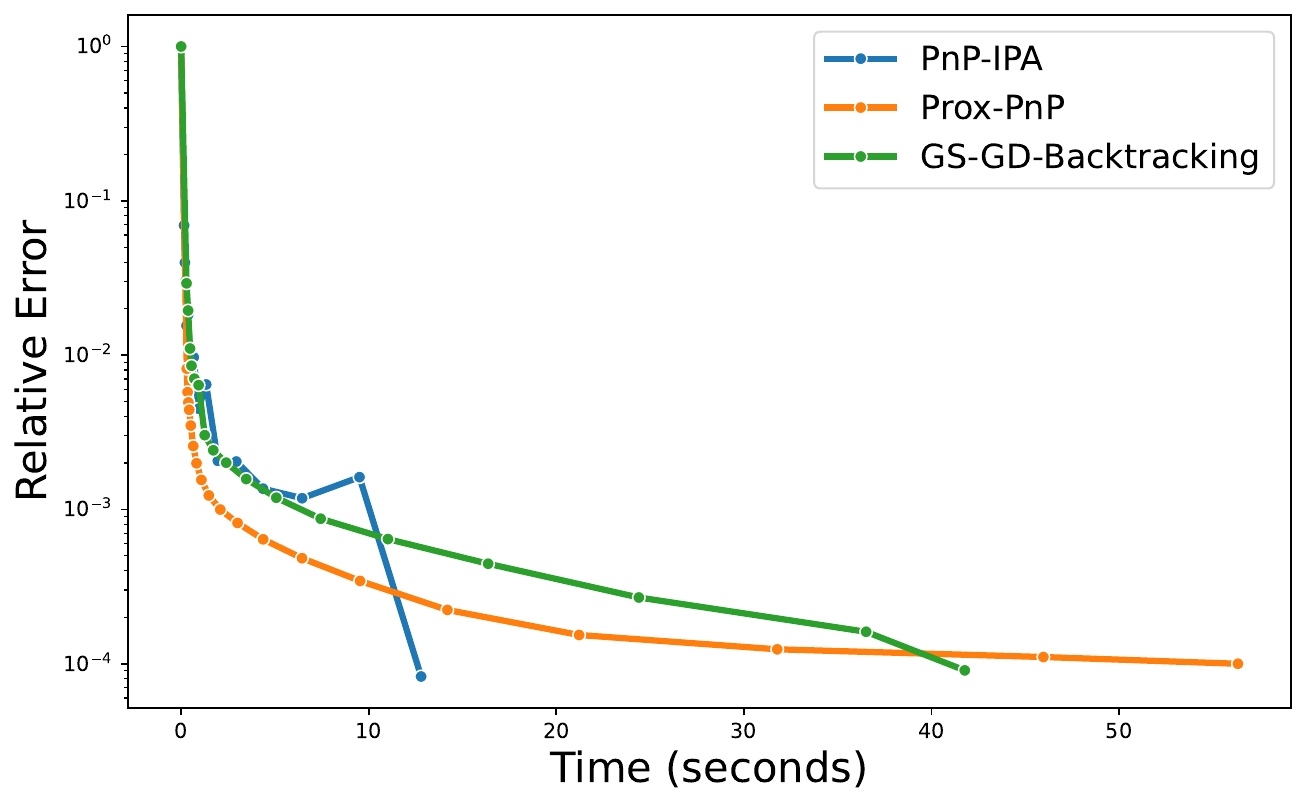}
        \caption{Relative error evolution along time.}
    \end{subfigure}
    
    \vspace{0.5cm} 
    

    {\renewcommand{\spysize}{1.6cm}%
    \begin{subfigure}{0.32\textwidth}
        \centering
        \spyimg{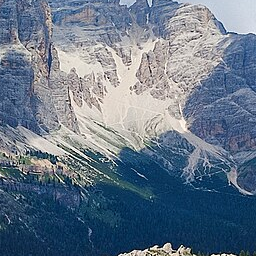}{1.80}{4.70}{0.78}{0.18}
        \caption{Clean image}
    \end{subfigure} \hfill
    \begin{subfigure}{0.32\textwidth}
        \centering
        \spyimg{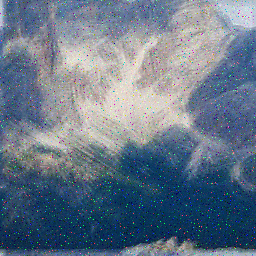}{1.80}{4.70}{0.78}{0.18}
        \caption{Blurred image}
    \end{subfigure} \hfill}
     \begin{subfigure}{0.32\textwidth}
        \centering
        \includegraphics[width=\linewidth]{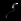}
        \vspace{0.17cm}
        \caption{Blur kernel ($6^{\text{th}}$ kernel of \cite{Levin-etal-2009})}
    \end{subfigure}

    \vspace{0.5cm}

    \renewcommand{\spysize}{1.6cm}%
    \begin{subfigure}{0.32\textwidth}
        \centering
        \spyimg{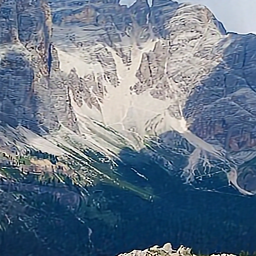}{1.80}{4.70}{0.78}{0.18}
        \caption{Reconstruction with PnP-IPA}
    \end{subfigure}\hfill
    \begin{subfigure}{0.32\textwidth}
        \centering
        \spyimg{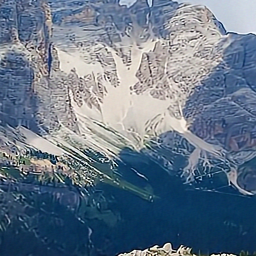}{1.80}{4.70}{0.78}{0.18}
        \caption{Recon. with GS-GD}
    \end{subfigure} \hfill
    \begin{subfigure}{0.32\textwidth}
        \centering
        \spyimg{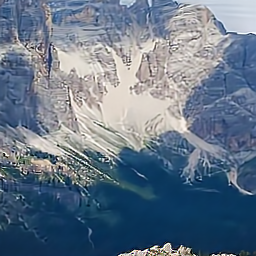}{1.80}{4.70}{0.78}{0.18}
        \caption{Recon. with Prox-PnP}
    \end{subfigure} \hfill
    \caption{Reconstruction results and convergence comparison under Cauchy noise ($\gamma = 0.01$). In this nonconvex setting with heavy-tailed noise, PnP-IPA outperforms both Prox-PnP and the Gradient Descent method with backtracking, achieving higher PSNR and faster convergence. The image is a patch from a photo of the Lagazuoi mountain, taken from the Wikimedia Commons repository \cite{Cinque-Torri-2026}.}
    \label{fig:cauchy_noise_figure}
\end{figure}

\subsection{Further remarks on PnP-IPA} \label{sec:IPA_remarks}

In this final subsection, we provide some further remarks on the observed convergence properties of our proposed algorithm.

First, both the number of inner steps and the number of backtracking reductions $m_k$ performed by the line search are low. More precisely, in Gaussian noise problems the inner solver performs 2 iterations on average, and always no more than 4, while for Cauchy noise problems only one inner iteration is sufficient to meet the stopping criterion. As for the backtracking loop, the number of backtracking reduction ranges from 0,1 for Gaussian low noise problems to 2 for Gaussian high and Cauchy noise. 
In particular, the latter information is important since a huge number of inner steps would have represented a bottleneck for the method, since every step requires an evaluation of the denoiser. This behaviour of the algorithm can be explained by the strong convexity of the inner objective function. 

Secondly, our objective is to motivate the choice of the step size $\alpha_k$, as detailed in Section \ref{sec:implementation_details}. To begin with, we can rewrite the forward step of our method as
\begin{equation}\nonumber
    z^k = x^k - \alpha_k \nabla f_0(x^k) = (1 + \alpha_k \lambda) x^k - \alpha_k \nabla f_{data} (x^k).
\end{equation}
Hence, the inexact proximal operator of $\hsigma / (\alpha_k \lambda)$ is computed at the point
\begin{equation}\nonumber
    \frac{z^k}{\alpha_k \lambda} = \left( 1 + \inv{\alpha_k \lambda} \right) x^k - \inv{\lambda} \nabla f_{data} (x^k).
\end{equation}
When $\alpha_k$ is large, the term $\inv{\alpha_k \lambda}$ is small. Therefore, the proximal operator $\mathrm{prox}_{\hsigma/(\alpha_k \lambda)}(\frac{z^k}{\alpha_k \lambda})$ approaches the identity mapping, as the function $\hsigma/(\alpha_k \lambda)$ tends to zero. Consequently, we obtain $\tilde{p}^k \approx \frac{z^k}{\alpha_k \lambda}$, where we can also neglect the term $\inv{\alpha_k \lambda} x^k$.

As a result, the update $\tilde{y}^k = D_\sigma(\tilde{p}^k)$ of our method essentially reduces to a pure PnP proximal gradient step with step size $\tau = 1 / \lambda$:
\begin{equation} \label{eq:PnP_approximation}
    \tilde{y}^k \approx D_\sigma \left( x^k - \inv{\lambda} \nabla f_{data}(x^k) \right).
\end{equation}
Naturally, the overall algorithm is not equivalent to Prox-PnP, also due to the presence of the line-search procedure. Moreover, our method theoretically converges for any choice of $\lambda$. We observe that for a small $\lambda$, a larger $\alpha_k$ is needed to reduce the term $\frac 1{\alpha_k\lambda}$ and obtain the approximation in \eqref{eq:PnP_approximation}. Thus, even if the choice of a large initial value $\alpha_0 = 10^6$ might appear unusual, in these specific settings the parameter $\alpha_k$ acts less like a traditional step size and more as a parameter regulating how closely the inexact proximal step approximates a pure PnP step of the form \eqref{eq:PnP_approximation}. The order of magnitude of $\alpha_0$ in the numerical experiments was set by manual adjustment. We then let $\alpha_k$ decrease to a fixed value, which we found beneficial to stabilize the method in its final iterations.

\section{Conclusions} \label{sec:conclusions}

In this paper, we introduced PnP-IPA (Plug-and-Play Inexact Proximal Algorithm), a novel and provably convergent optimization scheme for solving imaging inverse problems using a deep neural network-based regularizer. While the Plug-and-Play framework has empirically dominated the field of computational imaging, previous approaches providing rigorous theoretical guarantees have suffered from severe limitations, such as strict upper bounds on the regularization parameter, fixed step--sizes, and the requirement of convex data-fidelity terms.

To overcome these bottlenecks, we built upon the analytical properties of the Gradient-Step (GS) denoiser and designed a novel splitting strategy which allows for the inexact proximal computation of the scaled regularizer, hence enabling the use of variable step--sizes. Furthermore, to enable an adaptive step--size selection without exact evaluations of the objective function, we introduced a novel surrogate merit function that successfully guides an Armijo-like backtracking line--search. 

From a theoretical standpoint, we embedded our algorithm within the abstract KL framework, proving global convergence of the generated sequence to a stationary point of the nonconvex objective. Crucially, this result is achieved without any assumption on the regularization parameter $\lambda$ and fully encompasses nonconvex data-fidelity terms.

Numerical experiments on image deblurring tasks confirmed the practical superiority of PnP-IPA. By lifting the theoretical constraints on $\lambda$, our method allows for optimal parameter tuning, yielding state-of-the-art restoration quality and competitive computational times under both aggressive and weak Gaussian noise. Moreover, the robustness and flexibility of PnP-IPA were strongly highlighted in the nonconvex setting of Cauchy noise, where our algorithm significantly outperformed existing provable PnP methods that struggle with the extremely large Lipschitz constants of the data-fidelity gradient.

Future research directions will focus on proving that a single inner step is sufficient to satisfy the inexact proximal point condition, which would lead to a much more efficient algorithm. Moreover, we will explore the extension of our method to other denoisers beyond the Gradient-Step, such as Flow-Matching denoisers, thereby integrating a renoising procedure into our framework.

\ack{All the authors are members of the Gruppo Nazionale per il Calcolo Scientifico (GNCS) of the Italian Istituto Nazionale di Alta Matematica (INdAM), which is
kindly acknowledged.}



\data{The CBSD68 dataset used in the numerical experiments with Gaussian noise can be downloaded e.g. from https://huggingface.co/datasets/deepinv/CBSD68. The image is a patch from a photo of the Lagazuoi mountain, taken from the Wikimedia Commons repository \cite{Cinque-Torri-2026}.}


\appendix

\section{Proofs for the convergence analysis} \label{sec:proofs}

\begin{proof}[Proof of Lemma \ref{lem:squaredNorms}]
    We first recall that $H_k$ is $\theta$--strongly convex, $\hat y^k$ is its unique minimizer. Therefore, we have 
    \begin{equation*}
        \frac{\theta}{2} \norm{\hat{y}^k - \tilde y^k}^2 \leq H_k(\tilde y^k)- H_k(\hat{y}^k) \leq  - \frac{\mu}{2} \tkd{\Delta} 
    \end{equation*}
    which directly gives \eqref{eq:1epsilon}. Regarding \eqref{eq:2epsilon}, since $H_k(x^k) = 0$ we also have
    \begin{equation*}
        \frac{\theta}{2} \norm{\hat{y}^k - x^k}^2 \leq - H_k(\hat{y}^k) \leq - H_k(\tk{y}) - \frac{\mu}{2} \tkd{\Delta} \leq - \left ( 1 + \frac{\mu}{2} \right ) \tkd{\Delta},
    \end{equation*}
    where the second inequality is obtained by \eqref{eq:basicHYtilde} and the third by \eqref{eq:deltatildeineq}. Finally, regarding \eqref{eq:3epsilon}, by triangle inequality we compute
     \begin{align*}
        \norm{\tilde{y}^k - x^k} ^2 &\leq 2 \norm{\tilde{y}^k - \hat{y}^k}^2 + 2\norm{\hat{y}^k - x^k}^2 \leq - 2\frac\mu\theta \tkd{\Delta} - \frac{4}{\theta} \left ( 1 + \frac{\mu}{2} \right ) \tkd{\Delta},
    \end{align*}
where the second inequality is given by \eqref{eq:1epsilon}--\eqref{eq:2epsilon} and is equivalent to \eqref{eq:3epsilon}.
\end{proof}

\begin{proof}[Proof of Lemma \ref{lem: inequalities}]
By applying the Descent Lemma to $f_0$, we have
\begin{equation} \label{eq:f0lem2}
    f_0(\hat{y}^k) \geq f_0(\tilde{y}^k) - \inner{\nabla f_0(\hat{y}^k)}{\tilde{y}^k - \hat{y}^k} - \frac{L_{f_0}}{2} \norm{\tilde{y}^k - \hat{y}^k}^2.
\end{equation}

Inequality \eqref{eq:basicHYtilde} yields
\begin{equation}\nonumber
\begin{aligned}
f_1(\tilde y^k) + &\langle\nabla f_0(x^k),\tilde y^k-x^k\rangle + \frac{1}{2\alpha_k}\|\tilde y^k-x^k\|^2\\
& \leq f_1(\hat y^k) + \langle\nabla f_0(x^k),\hat y^k-x^k\rangle + \frac{1}{2\alpha_k}\|\hat y^k-x^k\|^2 -\frac{\mu}2\tilde \Delta_k
\end{aligned}
\end{equation}
which, rearranging the terms and neglecting the non--negative ones at the right-hand-side results in
\begin{equation} \label{eq:f1lem2new}
    \begin{aligned}
        f_1 (\hat{y}^k) &\geq f_1(\tk{y}) +   \langle \nabla f_0(x^k),\tilde y^k-\hat y^k\rangle -\frac{1}{2\alpha_k}\|\hat y^k-x^k\|^2 +\frac \mu 2 \tilde \Delta_k.
    \end{aligned}
\end{equation}
Summing \eqref{eq:f0lem2} with \eqref{eq:f1lem2new} gives
\begin{equation}\label{eq:A5}
F(\hat y^k)  \geq  F(\tilde y^k) + \langle \nabla f_0(x^k) -\nabla f_0(\hat y^k), \tilde y^k-\hat y^k\rangle - \frac{1}{2\alpha_k}\|\hat y^k-x^k\|^2 -\frac{L_{f_0}} 2 \|\tilde y^k-\hat y^k\|^2+\frac \mu 2 \tilde \Delta_k.
\end{equation}
Now we derive bounds for the second term in the previous inequality using the Cauchy-Schwarz inequality:
\begin{align*} 
     \inner{\nabla f_0 (x^k) - \nabla f_0(\hat{y}^k)}{\tk{y} - \hat{y}^k} & \leq \norm{\nabla f_0 (x^k) - \nabla f_0(\hat{y}^k)} \norm{\tk{y} - \hat{y}^k} \\
     & \leq L_{f_0} \norm{\hat{y}^k - x^k} \norm{\tk{y} - \hat{y}^k} \\
     & \leq L_{f_0} \sqrt{- \frac{2}{\theta} \left ( 1 + \frac{\mu}{2} \right ) \tkd{\Delta}} \cdot \sqrt{-\frac\mu\theta \tkd{\Delta}} \\
     & = \frac{L_{f_0}}{\theta} \sqrt{2\mu \left ( 1 + \frac{\mu}{2} \right )} (- \tkd{\Delta}),
\end{align*}
where the last inequality follows from Lemma \ref{lem:squaredNorms}. 
By exploiting the last inequality and \eqref{eq:1epsilon}--\eqref{eq:2epsilon}, equation \eqref{eq:A5} becomes
\begin{equation*}
    \begin{aligned}
        F(\hat{y}^k) &\geq F(\tilde{y}^k) - \frac{L_{f_0}}{\theta} \sqrt{2\mu \left ( 1 + \frac{\mu}{2} \right )} (- \tkd{\Delta}) - \frac{1}{\theta\alpha_k}\left(1+\frac \mu 2\right) (- \tkd{\Delta}) - \frac{L_{f_0}\mu}{2\theta} (- \tkd{\Delta}) + \frac{\mu}{2} \tkd{\Delta} \\
        &\geq F(\tilde{y}^k) + \left(\frac{L_{f_0}}{\theta} \sqrt{2\mu \left ( 1 + \frac{\mu}{2} \right )} + \frac{1}{\theta\alpha_{\min}} \left(1+\frac \mu 2\right) + \frac{L_{f_0}\mu}{2\theta} + \frac{\mu}{2} \right)\tkd{\Delta},
    \end{aligned}
\end{equation*}
which gives us \eqref{eq:firstLem2} with
\begin{equation*}
    c = \frac{L_{f_0}}\theta \sqrt{{2}\mu \left ( 1 + \frac{\mu}{2} \right )} + \frac{1}{\theta\alpha_{\min}} \left(1+\frac \mu 2\right) + \frac{L_{f_0}\mu}{2\theta} + \frac{\mu}{2}.
\end{equation*}
As for \eqref{eq:secondLem2}, using again the Descent Lemma we obtain
\[
    f_0(\hat{y}^k) \leq f_0(x^k) + \inner{\nabla f_0(x^k)}{\hat{y}^k - x^k} + \frac{L_{f_0}}{2} \norm{\hat{y}^k - x^k}^2.
\]
Summing $f_1(\hat{y}^k)$ on both sides yields
\begin{align*}
    F(\hat{y}^k) &\leq f_1(\hat{y}^k) + f_0(x^k) + \inner{\nabla f_0(x^k)}{\hat{y}^k - x^k} + \frac{L_{f_0}}{2} \norm{\hat{y}^k - x^k}^2 \\
    &= F(x^k) + f_1(\hat{y}^k) - f_1(x^k) + \inner{\nabla f_0(x^k)}{\hat{y}^k - x^k} + \frac{L_{f_0}}{2} \norm{\hat{y}^k - x^k}^2 \\
    &\leq F(x^k) + \underbrace{\inner{\nabla f_0(x^k)}{\hat{y}^k - x^k} + \inv{2 \alpha_k} \norm{\hat{y}^k - x^k}^2 + f_1(\hat{y}^k) - f_1(x^k)}_{= \ H_k(\hat{y}^k) \ \leq \ H_k(x^k) \ = \ 0} + \frac{L_{f_0}}{2} \norm{\hat{y}^k - x^k}^2 \\
    &\leq F(x^k) + \frac{L_{f_0}}{2} \norm{\hat{y}^k - x^k}^2 \\
    &\leq F(x^k) - \frac{L_{f_0}}{\theta} \left ( 1 + \frac{\mu}{2} \right ) \tkd{\Delta},
\end{align*}
where last inequality holds since \eqref{eq:2epsilon}. Inequality \eqref{eq:secondLem2} is then satisfied with $d = \frac{L_{f_0}}{\theta} \left ( 1 + \frac{\mu}{2} \right )$.
\end{proof}

\begin{proof}[Proof of Lemma \ref{lem:subgrad_bound}]
Applying Lemma \ref{lem:subcalculus_basic} to the function $H_k$ and recalling that $\hat y^k$ is its minimizer, in view of Definition \ref{def:subdiff} one obtains
\begin{equation}
0 \in \partial H_k(\hat{y}^k) =  \ \inv{\alpha_k} (\hat{y}^k - x^k ) + \nabla f_0(x^k) + \partial f_1(\hat{y}^k) \Longleftrightarrow - \inv{\alpha_k} ( \hat{y}^k - x^k ) - \nabla f_0(x^k) \in \partial f_1(\hat{y}^k), 
\end{equation}
i.e., $\hat{w}^k\in \partial f_1(\hat{y}^k)$. Since $\hat{v}^k = \hat{w}^k + \nabla f_0(\hat{y}^k)$, we have that $\hat{v}^k \in \nabla f_0(\hat{y}^k) + \partial f_1(\hat{y}^k) = \partial F(\hat{y}^k)$; moreover, using the triangular inequality, the Lipschitz continuity of $\nabla f_0$ and inequality \eqref{eq:2epsilon}, it holds that
    \begin{equation}
        \begin{aligned}
        \norm{\hat{v}^k} &\leq \inv{\alpha_k} \norm{\hat{y}^k - x^k} + \norm{- \nabla f_0(x^k) + \nabla f_0(\hat{y}^k)}
        \leq \inv{\alpha_k} \norm{\hat{y}^k - x^k} + L_{f_0} \norm{\hat{y}^k - x^k}\\ &\leq \left ( \inv{\alpha_k} + L_{f_0} \right ) \sqrt{- \frac{2}{\theta} \left ( 1 + \frac{\mu}{2} \right ) \tkd{\Delta}}
        \leq \left ( \inv{\alpha_{\min}} + L_{f_0} \right ) \sqrt{\frac{2}{\theta} \left ( 1 + \frac{\mu}{2} \right ) } \sqrt { - \tkd{\Delta}},
        \end{aligned}
    \end{equation}
which corresponds to \eqref{eq:subgrad_ineq} with $q = \left ( \inv{\alpha_{\min}} + L_{f_0} \right ) \sqrt{ \frac{2}{\theta} \left ( 1 + \frac{\mu}{2} \right ) }$.
\end{proof}

\begin{proof}[Proof of Proposition \ref{prop:H1}]
Inequality \eqref{eq:Phi_desc} follows from the line--search and from Lemma \ref{lemma:linesearch}. The inequality proves that the sequence $\Phi_k = \Phi(x^{k}, U^k)$ is decreasing. Moreover, since from Assumption [A4] $f_0$ is bounded from below, $\Phi$ is bounded from below as well. Therefore, summing up the inequality terms, \eqref{eq:Phi_desc} implies $-\sum_{k=0}^\infty \tilde{\Delta}_k < \infty$, which, in turn, yields $\lim_{k \to \infty} \tilde{\Delta}_k  = 0$. Recalling \eqref{eq:3epsilon}, this implies $\lim_{k \to \infty} \norm{\tilde{y}^{k} - x^{k}} = 0$. To conclude, the iteration update in the algorithm entails that 
\begin{equation*}
    \norm{x^{k+1} - x^{k}} \leq \norm{x^k + \eta_k d^k - x^k} = \eta_k \norm{d^k} \leq \norm{\tilde{y}^k - x^k},
\end{equation*}
which implies $\lim_{k \to \infty} \norm{x^{k+1} - x^{k}} = 0$.
\end{proof}

\begin{proof}[Proof of Proposition \ref{prop:H2}]
From inequality \eqref{eq:line_search_cons}, we obtain
\begin{equation*}
    \Phi(x^{k+1}, U^{k+1}) \leq  F(\tilde{y}^{k})\leq F(\hat{y}^{k}) - c \tkd{\Delta},
\end{equation*}
where the last inequality follows from \eqref{eq:firstLem2}. Setting $\rho_k$ as in \eqref{eq:rho_k}, 
we obtain the left-most inequality in \eqref{eq:H2_ineq}. On the other hand, inequality \eqref{eq:secondLem2} entails
\begin{equation*}
  F(\hat{y}^{k}) + \frac{1}{2} \rho_k^2 \leq F(x^{k}) - d \tkd{\Delta} + \frac{1}{2} \rho_k^2 = f_0(x^k) + f_1(x^{k}) - d \tkd{\Delta} + \frac{1}{2} \rho_k^2 \leq f_0(x^k) + U^k - d \tkd{\Delta} + \frac{1}{2} \rho_k^2.
\end{equation*}
Setting $r_k$ as in \eqref{eq:r_k} gives the rightmost inequality in \eqref{eq:H2_ineq}. Finally, from \eqref{eq:limits} we have that $\lim_{k \to \infty} r_k = \lim_{k \to \infty} \rho_k = 0$.
\end{proof}

\section{Hyperparameters of the baseline algorithms} \label{sec:baselines_hyperparameters}

For the sake of reproducibility, this appendix collects in
Table~\ref{tab:baselines_hyperparameters} the hyperparameter setting used for
the three Hurault-based baselines compared against PnP-IPA in
Section~\ref{sec:numerical_experiments}. The
values mostly follow the guidelines of the original references
\cite{Hurault-etal-2022a, Hurault-etal-2024}, except for the weak Gaussian noise for which a different combination of hyperparameters has been found to produce better results. We remind that Relaxed-Prox-PnP is
omitted from the high-noise regime since its optimal $\gamma_{\text{relax}}$
in that setting is $1$, making the algorithm coincide with Prox-PnP.

\begin{table}[h!]
\centering
\caption{Hyperparameters of the baseline algorithms across the three
experimental regimes. $\nu$ denotes the standard deviation of the Gaussian
noise, $\gamma_{\text{relax}}$ denotes the relaxation parameter; a dash
indicates a parameter that is not used by the corresponding method.}
\label{tab:baselines_hyperparameters}
\begin{tabular}{lcccc}
\toprule
Method / Regime & $\lambda$ & $\alpha$ & $\gamma_{\text{relax}}$ & $\sigma_{\text{den}}$ \\
\midrule
\multicolumn{5}{l}{\textit{Prox-PnP}} \\
\quad high Gaussian ($\nu = 0.05$) & $1.00$               & --- & ---   & $0.50\,\nu$    \\
\quad weak Gaussian ($\nu = 0.01$) & $1.49$               & --- & ---   & $0.75\,\nu$    \\
\midrule
\multicolumn{5}{l}{\textit{Relaxed-Prox-PnP}} \\
\quad weak Gaussian ($\nu = 0.01$) & $1.625$              & ---     & $0.6$ & $1.25\,\nu$ \\
\midrule
\multicolumn{5}{l}{\textit{alpha-Prox-PnP}} \\
\quad high Gaussian ($\nu = 0.05$) & $1.00$               & $0.50$  & $1.0$ & $0.50\,\nu$ \\
\quad weak Gaussian ($\nu = 0.01$) & $2.66$               & $0.37$  & $0.6$ & $1.25\,\nu$ \\
\bottomrule
\end{tabular}
\end{table}

\bibliographystyle{plain}
\bibliography{references}

\end{document}